\documentclass[a4paper]{amsart}

\usepackage{amsmath,amsfonts,amssymb,amsthm} 

\theoremstyle{definition}
\newtheorem{defn}{Definition}[section]

\newtheorem{notation}[defn]{Notation}
\newtheorem{construct}[defn]{Construction}

\theoremstyle{plain}
\newtheorem{lemma}[defn]{Lemma}
\newtheorem{theorem}[defn]{Theorem}
\newtheorem{prop}[defn]{Proposition}
\newtheorem{cor}[defn]{Corollary}

\newtheorem{innercustomgeneric}{\customgenericname}
\providecommand{\customgenericname}{}
\newcommand{\newcustomtheorem}[2]{%
  \newenvironment{#1}[1]
  {%
   \renewcommand\customgenericname{#2}%
   \renewcommand\theinnercustomgeneric{##1}%
   \innercustomgeneric
  }
  {\endinnercustomgeneric}
}

\newcustomtheorem{customthm}{Theorem}
\newcustomtheorem{customcor}{Corollary}

\theoremstyle{remark}
\newtheorem{rk}[defn]{Remark}

\newcommand{\PS}{\mathbb{P}}
\newcommand{\F}{\mathbb{F}}
\newcommand{\Z}{\mathbb{Z}}

\newcommand{\C}{\mathbb{C}}
\newcommand{\A}{\mathbb{A}}
\newcommand{\D}{\mathbb{D}}
\newcommand{\R}{\mathbb{R}}

\usepackage{todonotes}

\usepackage{subcaption}
\usepackage{mathrsfs} 
\usepackage[shortlabels]{enumitem} 
\usepackage{pict2e} 
\usepackage{caption} 
\usepackage[T1]{fontenc} 

\usepackage[a4paper,left=2.8cm,right=2.8cm, bottom=4cm, top=3cm]{geometry}
\usepackage[colorlinks=true,citecolor=blue]{hyperref} 
\usepackage{mathtools} 
\mathtoolsset{centercolon} 

\usepackage{tikz}
\usetikzlibrary{cd}
\usetikzlibrary{patterns}
\usepackage{subcaption}
\usepackage{placeins}
\usepackage{mathabx}
\usepackage{kbordermatrix}
\usepackage{multirow}

\begin{document}
\bibliographystyle{alpha}

\title{Type II Degenerations of K3 Surfaces of Degree 4}

\author[J. M. Jones]{James Matthew Jones}
\address{Department of Mathematical Sciences, Loughborough University, Loughborough, Leicestershire, LE11 3TU, United Kingdom.}
\email{J.Jones@lboro.ac.uk}
\thanks{The author was partially supported by the EPSRC grant EP/T518098/1.}
\thanks{For the purpose of open access, the author has applied a Creative Commons Attribution (CC BY) licence to any Author Accepted Manuscript version arising.}

\begin{abstract}
We study Type II degenerations of K3 surfaces of degree 4 where the central fiber consists of two rational components glued along an elliptic curve. Such degenerations are called Tyurin degenerations. We construct explicit Tyurin degenerations corresponding to each of the 1-dimensional boundary components of the Baily-Borel compactification of the moduli space of K3 surfaces of degree 4.
For every such boundary component we also construct an 18-dimensional family of Tyurin degenerations of K3 surfaces of degree 4 and compute the stable models of these degenerations.
\end{abstract}
\date{}
\maketitle
\section{Introduction}
\subsection{Motivation}

A Tyurin degeneration of K3 surfaces is a semistable degeneration $X\rightarrow \Delta $, where $\Delta$ is the complex unit disc, whose central fiber $X_0=V_0\cup V_1$ is the union of two rational surfaces which are glued along a smooth elliptic curve. Such a degeneration is a Tyurin degeneration of K3 surfaces of degree 4 if there is a line bundle $L$ on $X$ such that $L_t:=L|_{X_t}$ on the general fiber $X_t$ is nef and big, with $L_t^2=4$. By results of Shepherd-Barron \cite{SB83} one can further arrange that $L_0$ is nef by flopping curves on $X_0$.

This point of view is equivalent to considering K3 surfaces which are polarized by the lattice $\langle4\rangle$ and Tyurin degenerations occur over the 1-dimensional strata of the Baily-Borel compactification $\overline{\mathcal{F}_4}\supset\mathcal{F}_4$ of the moduli space of K3 surfaces of degree 4. These 1-dimensional strata have several lattice-theoretic descriptions associated to them. Let $\Lambda_4:= \langle-4\rangle\oplus H^{\oplus2}\oplus(E_8)^{\oplus2}$ and consider $I\subset \Lambda_4$ a rank 2 isotropic sublattice, then the 1-dimensional strata of $\overline{\mathcal{F}_4}$ are in bijective correspondence with isomorphism classes of $I^\perp/I$. Such $I^\perp/I$ are rank 17 lattices and Scattone \cite{Sca87} showed that there are 9 isomorphism classes of these lattices in the degree 4 case. When one knows the central fiber $X_0=V_0\cup V_1$ of a Tyurin degeneration explicitly, Friedman \cite{Fried84} was able to show that these rank 17 lattices can be obtained by considering the limit mixed Hodge structure of $X_0$.

Due to this, there are several natural questions one can ask. For example, is it possible to describe the geometric objects giving rise to these lattices explicitly? Do these geometric descriptions uniquely determine the lattices which one obtains? Are there distinct geometric descriptions which are able to determine the same lattice?

The first of these questions is made approachable due to previous work of Shah \cite{Shah81} where many singular quartic surfaces are constructed via methods in Geometric Invariant Theory (GIT). These surfaces can be extended to degenerations of K3 surfaces, and from Shah's analysis it is clear to see which surfaces can be used to define Type II degenerations of K3 surfaces of degree 4.

We construct one degeneration whose central fiber is not a quartic surface. In this case it appears to be more natural to construct a degeneration of hyperelliptic K3 surfaces of degree 4. This degeneration appears to be very special as discussed in work of Laza and O'Grady \cite{LOG19,LOG21} where the authors compare the GIT and Baily-Borel compactifications of $\mathcal{F}_4$ using variation of GIT. 

\subsection{Description of the Results}

The following is our first result. 

\begin{customthm}{1}\label{correspondence intro}
Let $X \rightarrow \Delta$ be a degeneration of K3 surfaces of degree 4 whose central fiber $X_0$ is a surface belonging to column 1 of Table \ref{tab:correspondence} and let $\widetilde{X}$ be the minimal resolution of $X$. Then $\widetilde{X}$ is a projective Tyurin degeneration of K3 surfaces of degree 4, whose central fiber belongs to column 2 of Table \ref{tab:correspondence}, and the associated Type II boundary component of $\overline{\mathcal{F}_4}$ is given by column 3 of Table \ref{tab:correspondence}.
\begin{table}[h]
\centering
\begin{tabular}{|c|l|l|c|}
\hline
\multirow{2}{*}{\textbf{Central fiber of $X$}}                                                                                                                                                    & \multicolumn{1}{c|}{\multirow{2}{*}{\textbf{Central fiber of $\widetilde{X}$}}} & \multicolumn{1}{c|}{\multirow{2}{*}{\textbf{Lattice}}} & \multirow{2}{*}{\textbf{Section}}                               \\                                                                                                                                & \multicolumn{1}{c|}{}                                                           & \multicolumn{1}{c|}{}                                  &                                                                 \\ \hline
\begin{tabular}[c]{@{}c@{}}Quartic surface with one singular point\\ of type $\widetilde{E_8}$ with a line passing through it\end{tabular}                                                        & $\operatorname{Bl}_{10}\PS^2\cup\operatorname{Bl}_8\PS^2$                       & $E_8^{\oplus2}+\langle -4 \rangle$                     & Section \ref{sect: E8 singularity and a line}  \\ \hline
\begin{tabular}[c]{@{}c@{}}Quartic surface with one singular point \\ of type $\widetilde{E_8}$ and no line passing through it\end{tabular}                                                       & $\operatorname{Bl}_{10}\PS^2\cup\operatorname{Bl}_8\PS^2$                       & $E_8 +D_9$                                             & Section \ref{sect: E8 singularity and no line} \\ \hline
\begin{tabular}[c]{@{}c@{}}Quartic surface with one singular point \\ of type $\widetilde{E_7}$\end{tabular}                                                                                      & $\operatorname{Bl}_{11}\PS^2\cup\operatorname{Bl}_7\PS^2$                       & $E_7^{\oplus2}+A_3$                                    & Section \ref{sect: E7 singularity}             \\ \hline
A plane intersecting a cubic surface                                                                                                                                                              & $\operatorname{Bl}_{12}\PS^2\cup\operatorname{Bl}_6\PS^2$                       & $E_6+A_{11}$                                           & Section \ref{sect:Plane and Cubic}             \\ \hline
\begin{tabular}[c]{@{}c@{}}A double cover of $\PS(1,1,2)$ branched\\ along a doubled twisted cubic passing \\ through the vertex and along a conic \\ not passing through the vertex\end{tabular} & $\operatorname{Bl}_{18}\PS^2\cup\PS^2$                                          & $D_{17}$                                               & Section \ref{sect:hyperelliptic degen}         \\ \hline
\begin{tabular}[c]{@{}c@{}}Quartic surface non-normal along\\ a twisted cubic\end{tabular}                                                                                                        & $\operatorname{Bl}_{17}\PS^2\cup (\PS^1\times\PS^1)$                            & $D_{16}+\langle -4 \rangle$                            & Section \ref{sect:non-normal along TC}         \\ \hline
\begin{tabular}[c]{@{}c@{}}Quartic surface non-normal along\\ a conic\end{tabular}                                                                                                                & $\operatorname{Bl}_{13}\PS^2\cup\operatorname{Bl}_5\PS^2$                       & $D_{12}+D_5$                                           & Section \ref{sect:non-normal along conic}      \\ \hline
\begin{tabular}[c]{@{}c@{}}Quartic surface non-normal along\\ a line\end{tabular}                                                                                                         & $\operatorname{Bl}_9\PS^2\cup\operatorname{Bl}_9\PS^2$                          & $D_8^{\oplus2}+\langle -4 \rangle$                     & Section \ref{sect:non-normal along line}       \\ \hline
Two quadrics intersecting transversally                                                                                                                                                           & $\operatorname{Bl}_{16}(\PS^1\times\PS^1)\cup(\PS^1\times\PS^1)$                & $A_{15}+A_1^{\oplus2}$                                 & Section \ref{sect:Two Quadrics}                \\ \hline
\end{tabular}
    \caption{The correspondence between the central fibers of singular degenerations of K3 surfaces of degree 4, the central fibers of their minimal resolutions, and the Type II boundary components of $\overline{\mathcal{F}_4}$.}\label{tab:correspondence}
    \vspace{-0.35cm}
    \end{table}

\end{customthm}

Theorem \ref{correspondence intro} gives explicit descriptions of the Tyurin degenerations which lie over the boundary components of $\overline{\mathcal{F}_4}$ and in Section \ref{sect:Explicit descriptions of families} we provide equations for such degenerations. We also describe the two components, $V_0$ and $V_1$, of the central fiber $\widetilde{X}_0$ of the corresponding Tyurin degenerations. An interesting consequence of this is that the isomorphism classes of $V_0$ and $V_1$ do not uniquely determine the class of Tyurin degeneration which one obtains.

In each case $V_0$ and $V_1$ are obtained by blowing up either $\PS^2$ or $\PS^1\times\PS^1$ at points on an elliptic curve $E$. In order for $X_0$ to be the central fiber of a degeneration, these points must satisfy a relation in $\operatorname{Pic}(E)$ which arises due to $d$-semistability \cite{Fried84,AE23}. In Section \ref{sect:period domains} we show that for each explicitly constructed Tyurin degeneration in Section \ref{sect:Explicit descriptions of families}, the points satisfy an additional relation in $\operatorname{Pic}(E)$. It is also shown that the period map for Tyurin degenerations imposes this relation. We provide the following construction which is necessary for Theorem \ref{list of nef models}.

\begin{construct}\label{Tyurin constructions}
Fix $E$ an elliptic curve, let $\overline{V_0} = \PS^2 \text{ or }\PS^1\times\PS^1$ and let $\overline{V_1} = \PS^2 \text{ or }\PS^1\times\PS^1$. Set $k_0=(-K_{\overline{V_0}})^2$, $k_1= (-K_{\overline{V_1}})^2$ and define $k=k_0+k_1$. Embed $q\in E \hookrightarrow \overline{V_0}$ and $q'\in E\hookrightarrow \overline{V_1}$ so that $q$ and $q'$ are the identities of the respective group laws. Fix $k$ points $p_1,\dots, p_{n}, p_1',\dots, p_{k-n}'\in E$ which satisfy the relation
$$\mathcal{O}_E\left(k_0q+k_1q'-\sum_{i=1}^np_i-\sum_{i=1}^{k-n}p_i'\right)=\mathcal{O}_E.$$
This is the relation imposed by $d$-semistability. Define $V_0 = \operatorname{Bl}_n\overline{V_0}$ and $V_1 = \operatorname{Bl}_{k-n}\overline{V_1}$ where $\overline{V_0}$ is blown up in the points $p_1,\dots,p_n$ and $\overline{V_1}$ is blown up in the points $p_1',\dots,p_{k-n}'$. We set $d=n-k_0$ and glue $V_0$ and $V_1$ along $E$.
\end{construct}

The relations described in the final column of Table \ref{Complete Description Table} are shorthand for the appropriate relation in $\operatorname{Pic}(E)$. For example, the relation in the final row is shorthand for $\mathcal{O}_E(16q-\sum_{i=1}^{16}p_i)=\mathcal{O}_E$. These relations make it possible to distinguish the 9 degenerations from a geometric point of view, we use them and the following definitions of Alexeev and Engel \cite{AE23} to provide a classification of Tyurin degenerations.

\begin{defn}
Let $X^*\rightarrow\Delta^*:=\Delta\setminus\{0\}$ be a family of smooth K3 surfaces over the punctured unit disc and let $L^*$ be a line bundle on $X^*$ which is relatively nef and big over $\Delta^*$. A relatively nef extension $L$ to a Kulikov model $X$ over $\Delta$ is called a \emph{nef model}. Let $H^*\subset X^*$ be the vanishing locus of a section of $L^*$ containing no vertical components. A \emph{divisor model} is an extension $H \subset X$ to a relatively nef divisor $H\in|L|$ for which $H_0$ contains no strata of $X_0$. A \emph{stable model} of $(X^*,H^*)$ model is the log canonical model of some divisor model $(X,H)$.
\end{defn}

The existence of nef models is originally due to Shepherd-Barron \cite{SB83} and the existence of divisor models is due to Laza \cite{Laza16}. In Section \ref{sect:clasification of models} we determine all stable models for the central fibers of Tyurin degenerations of K3 surfaces of degree 4. Combining this with the results of Section \ref{sect:period domains} allows us to distinguish them using their geometry.

\begin{customthm}{2}\label{list of nef models}
Let $X\rightarrow\Delta$ be a Tyurin degeneration of K3 surfaces of degree 4  with $X_0=V_0\cup V_1$ as given by Construction \ref{Tyurin constructions}. Let $L$ be the polarization line bundle on $X$ so that $L_t^2=4$ for all $t\in \Delta$, $L_t$ is ample on the general fiber and $L_0$ is nef. Let $(X,H)$ be a divisor model where  $H_t\in|L_t^{\otimes n}|$ for all $t$. If $H$ is relatively ample then the central fiber $\overline{X_0}$ of the stable model $\left(\overline{X},\overline{H}\right)$ is one of the possibilities given in Table \ref{Complete Description Table} and all such possibilities occur. If $H$ is relatively nef but not relatively ample, then $\overline{X_0}$ is obtained from Table \ref{Complete Description Table} by contracting curves or a component.

\begin{table}[h]
\centering
\begin{tabular}{|l|c|c|c|l|}
\hline
\multicolumn{1}{|c|}{\textbf{Lattice}} & $V_0$                                      & $V_1$                      & $d$ & \multicolumn{1}{c|}{\textbf{Additional Relation in $\operatorname{Pic}(E)$}} \\ \hline
$E_8+E_8+\langle -4 \rangle$           & $\operatorname{Bl}_9\PS^2$                 & $\operatorname{Bl}_9\PS^2$ & 0   & $27q=3\sum_{i=1}^{8}p_i+2p_9+p_9'$                                           \\ \hline
$E_8+E_8+\langle -4 \rangle$           & $\operatorname{Bl}_{10}\PS^2$              & $\operatorname{dP}_1$      & 1   & $27q=3\sum_{i=1}^{8}p_i+2p_9+p_{10}$                                        \\ \hline
$E_8+D_9$                              & $\operatorname{Bl}_{10}\PS^2$              & $\operatorname{dP}_1$      & 1   & $21q = 3p_1+2\sum_{i=2}^{10}p_i$                                             \\ \hline
$E_7+E_7+A_3$                          & $\operatorname{Bl}_{11}\PS^2$              & $\operatorname{dP}_2$      & 2   & $18q=2\sum_{i=1}^{7}p_i+\sum_{i=8}^{11}p_i$                                  \\ \hline
$E_6+A_{11}$                           & $\operatorname{Bl}_{12}\PS^2$              & $\operatorname{dP}_3$      & 3   & $12q=\sum_{i=1}^{12}p_i$                                                     \\ \hline
$E_6+A_{11}$                           & $\operatorname{Bl}_{18}\PS^2$              & $\PS^2$                    & 9   & $12q=\sum_{i=1}^{12}p_i$                                                     \\ \hline
$D_{17}$                               & $\operatorname{Bl}_{18}\PS^2$              & $\PS^2$                    & 9   & $45q=11p_1+2\sum_{i=2}^{18}p_i$                                              \\ \hline
$D_{16}+\langle -4\rangle$             & $\operatorname{Bl}_{17}\PS^2$              & $\PS^1\times\PS^1$         & 8   & $63q=15p_1+3\sum_{i=2}^{17}p_i$                                              \\ \hline
$D_{12}+D_5$                           & $\operatorname{Bl}_{13}\PS^2$              & $\operatorname{dP}_4$      & 4   & $15q = 3p_1+\sum_{i=2}^{13}p_i$                                              \\ \hline
$D_8+D_8+\langle-4\rangle$             & $\operatorname{Bl}_9\PS^2$                 & $\operatorname{Bl}_9\PS^2$ & 0   & $12q'+p_1=3q+2p_1'+\sum_{i=2}^9p_i'$                                         \\ \hline
$A_{15}+A_1+A_1$                       & $\operatorname{Bl}_{16}(\PS^1\times\PS^1)$ & $\PS^1\times\PS^1$         & 8   & $16q=\sum_{i=1}^{16}p_i$                                                    \\ \hline
\end{tabular}
\caption{The possibilities for a central fiber $\overline{X_0}$ of a stable model $(\overline{X},\overline{H})$ of a Tyurin degeneration of K3 surfaces of degree 4 when $H$ is relatively ample.}
\label{Complete Description Table}
\end{table}
\end{customthm}

In Section \ref{sect:clasification of models} we treat the relatively nef case of Theorem \ref{list of nef models} more carefully and explicitly describe the curves or the component which one must contract to obtain the relevant stable model. As a consequence of Theorem \ref{list of nef models} we can see that no two distinct degenerations of Theorem \ref{correspondence intro} incur the same additional relation in $\operatorname{Pic}(E)$. Therefore the additional relation is able to completely distinguish the Type II boundary components of $\overline{\mathcal{F}_4}$. Furthermore, we see that there are only two Type II boundary components of $\overline{\mathcal{F}_4}$ which admit distinct stable models $\left(\overline{X},\overline{H}\right)$ when $H$ is relatively ample.

Our main corollary of Theorem \ref{list of nef models} is a statement about \emph{$\lambda$-families} $\mathcal{X}\rightarrow S_\lambda$ (Definition \ref{def:lambda family}). These families were introduced by Alexeev and Engel \cite{AE23} and can be identified with families over a tubular neighbourhood of the toroidal compactification of $\mathcal{F}_4$. 

\begin{customcor}{3}\label{thm:lambda families intro}
Let $\mathcal{X}\rightarrow S_\lambda$ be a $\langle 4 \rangle$-quasipolarized nef $\lambda$-family of Tyurin degenerations and let $(X\rightarrow\Delta,H)$ be a divisor model arising from the $\lambda$-family. Then the stable model of $(X,H)$ is one of those constructed in Theorem \ref{list of nef models}.
\end{customcor}

\subsection{Relation to Previous Work} 
The first to consider a correspondence similar to that of Table \ref{correspondence intro} was Friedman \cite{Fried84} who studied the degree 2 case. Indeed, in the degree 2 situation Friedman used results of Shah \cite{Shah80}, where semistable degenerations of double covers of $\PS^2$ branched along smooth sextic curves were studied using GIT. Friedman developed many ideas which we use during this paper, the most fundamental of which is the construction of the period map for Tyurin degenerations \cite{Fried84,Carlson85}. This map is used to determine the relations in Table \ref{Complete Description Table}.

The correspondence between Columns 1 and 3 of Table \ref{tab:correspondence} seems to have been broadly known by Laza and O'Grady \cite[Table 5]{LOG19}, where the period map from the GIT moduli space of quartic surfaces $\mathfrak{p}:\mathfrak{M}\dashrightarrow\overline{\mathcal{F}_4}$ was studied. However, explicit degenerations which give rise to the Type II boundary strata in the Baily-Borel compactification were not constructed. In this regard, the improvement we provide here is in performing explicit birational modifications to the singular quartic surfaces of Shah \cite{Shah81}, and therefore being able to describe column 2 of Table \ref{tab:correspondence}.

By examining the period map from $\mathfrak{M}\dashrightarrow \overline{\mathcal{F}_4}$, Laza and O'Grady determined that there is no singular quartic surface of Type II which gives rise to the $D_{17}$ stratum in the Baily-Borel compactification. Instead any singular quartic which gives rise to the $D_{17}$ boundary locus is contained in the indeterminacy locus of the map $\mathfrak{p}$.

In a related work, Laza and O'Grady \cite{LOG21} also considered the GIT compactification for double covers of $\PS^1\times \PS^1$ branched along $(4,4)$-divisors in a similar way. In \cite{LOG21} it is shown that the Type II boundary component $D_{17}$ in $\overline{\mathcal{F}_4}$ does not arise as a double cover of $\PS^1\times\PS^1$ branched along a $(4,4)$-divisor but is instead the double cover of $\PS(1,1,2)$ given in Table \ref{tab:correspondence}. Our improvement to this work is again performing the explicit calculations to determine a Type II degeneration corresponding to $D_{17}$, and explicitly describing its central fiber. These explicit calculations are necessary to determine the relations of Theorem \ref{list of nef models}.

Recently, Alexeev and Engel \cite{AE23} have extensively studied the moduli space of K3 surfaces and several of its compactifications. In particular they provide comparisons between toroidal and semi-toroidal compactifications, as well as with the Baily-Borel compactification. Alexeev and Engel also demonstrate the existence of a family of Type II Kulikov surfaces for which the period map is the identity. Using this, they introduce the notion of Type II $\lambda$-families, which are families over tubular neighbourhoods of the Type II boundary components of toroidal compactifications of $\mathcal{F}_4$. With this frame of reference, Corollary \ref{thm:lambda families intro} is a statement about the structure of Type II boundary components of toroidal compactifications of $\mathcal{F}_4$.
\subsection{Data Availability}
Data sharing not applicable to this article as no datasets
were generated or analysed during the current study.
\subsection{Acknowledgements}
The results of this paper comprise one chapter of my PhD thesis. I would like to thank my supervisor, Alan Thompson, for his assistance and support throughout this project.
\section{Background Material}\label{sect:Background}
We begin by reviewing fundamental results regarding K3 surfaces.

\begin{defn}
A \emph{K3 surface} $Y$ is a compact complex surface such that $h^1(Y)=0$ and $\omega_Y\simeq\mathcal{O}_Y$.
\end{defn}

\begin{defn}
Let $\Lambda_{K3} = H^{\oplus3}\oplus(E_8)^{\oplus2}$ be a fixed copy of the unique even unimodular lattice of signature $(3,19)$, where $H$ is the unique even unimodular lattice of rank 2. We call $\Lambda_{K3}$ the \emph{K3 lattice}.
\end{defn}

For a K3 surface $Y$, the second cohomology group $H^2(Y,\Z)$ together with the cup product pairing is a lattice isometric to the K3 lattice. For our purposes we will be interested in \emph{quasipolarized} K3 surfaces.

\begin{defn}
A \emph{$\langle 2k \rangle$-quasipolarized K3 surface} $(Y,h)$ is a pair consisting of a K3 surface $Y$ and a primitive nef and big class $h \in \operatorname{NS}(Y)$ with $(h,h)=2k$. Two such $(Y,h)$ and $(Y',h')$ are \emph{isomorphic} if there is an isomorphism $f:Y\rightarrow Y'$ of K3 surfaces such that $f^*(h')=h$.
\end{defn}

The following theorem is known as the strong Torelli theorem, which was proved in the algebraic setting by Pjatecki{\u i}-Shapiro
and Shafarevi{\v c} \cite[p.551]{PSS71}. The strong Torelli theorem was proved in greater generality by Burns and Rapoport \cite{BurnsRapoport}. For our work, it is sufficient to consider only the algebraic setting because $\langle 2k\rangle$-quasipolarized K3 surfaces are algebraic \cite[Remark 2]{HarThom15}.

\begin{theorem}[Strong Torelli]\cite{PSS71,BurnsRapoport}\label{hodge isometry}
Let $(Y,h)$ and $(Y',h')$ be two $\langle 2k \rangle$-quasipolarized K3 surfaces. Assume that there is a lattice isometry $\phi:H^2(Y',\Z)\rightarrow H^2(Y,\Z)$ whose $\C$-linear extension preserves the Hodge decomposition and satisfies $\phi(h')=h$. Then there is a unique isomorphism $f:Y\rightarrow Y'$ with $\phi=f^*$, so $(Y,h)$ and $(Y',h')$ are isomorphic.
\end{theorem}

The strong Torelli theorem for K3 surfaces enables us to construct a moduli space of $\langle 2k \rangle$-quasipolarized K3 surfaces. We first recall two more notions which are necessary to construct the moduli space of $\langle 2k \rangle$-quasipolarized K3 surfaces.

\begin{defn}
A \emph{marking} on $Y$ is a choice of isometry $\phi:H^2(Y,\Z)\rightarrow \Lambda_{K3}$ and we say that the pair $(Y,\phi)$ is a \emph{marked} K3 surface. Fix a primitive class $h\in \Lambda_{K3}$ with $(h,h)=2k$. A \emph{marked $\langle 2k \rangle$-quasipolarized K3 surface} is a marked K3 surface $(Y,\phi)$ such that $\phi^{-1}(h)$ is a nef and big class in $\operatorname{NS}(Y)$.
\end{defn}

\begin{defn}
The \emph{polarized K3 lattice of degree $2k$} is denoted by $\Lambda_{2k}=\langle -2k\rangle \oplus H^{\oplus 2}\oplus E_8^{\oplus 2}$.
\end{defn}

\begin{defn}
Fix a primitive class $h\in \Lambda_{K3}$ such that $(h,h)=2k$. The \emph{period domain for marked K3 surfaces of degree $2k$} is
$$\D_{2k} = \PS\{x \in \Lambda_{K3}\otimes\C \text{ }\big{|}\text{ } x\cdot x=0,x\cdot\bar{x}>0,x\cdot h=0\}.$$
$\D_{2k}$ is a 19-dimensional open analytic subset of $\PS^{21}$. Let $\mathcal{X}\rightarrow S$ be a family of marked K3 surfaces of degree $2k$, the period map is defined to be $P: S\rightarrow \D_{2k}$ which maps $s \mapsto H^{2,0}(X_s)\subset H^2(X_s,\C).$
\end{defn}

\begin{theorem}\cite[Corollary 6.4.3, Remark 6.4.5]{Huy16}
Let $\Gamma(h) = \{\gamma \in O(\Lambda_{K3}) \text{ } \big{|} \text{ } \gamma(h) = h\}$ then $\mathcal{F}_{2k}=\D_{2k}/\Gamma(h)$ is a coarse moduli space of quasipolarized K3 surfaces of degree $2k$. Furthermore, $\mathcal{F}_{2k}$ is a 19-dimensional quasi-projective variety. 
\end{theorem}

It is well known that $\mathcal{F}_{2k}$ admits a compactification called the \emph{Baily-Borel compactification}, which we denote by $\overline{\mathcal{F}_{2k}}$ \cite{BB66}. $\overline{\mathcal{F}_{2k}}$ is a $19$-dimensional projective algebraic variety and it is known that the boundary components of $\overline{\mathcal{F}_{2k}}$ are $0$- and $1$-dimensional. In this paper we concern ourselves only with degree 4 K3 surfaces and we are interested in the $1$-dimensional boundary components of $\overline{\mathcal{F}_{4}}$. We first define $\langle 2k \rangle$-quasipolarized Type II degenerations of K3 surfaces.

\begin{defn}
A \emph{Kulikov model} $\pi: X\rightarrow \Delta=\{t\in\C \text{ } \big{|} \text{ } |t|<1\}$ is a flat proper surjective morphism from a smooth threefold $X$ to the complex unit disc $\Delta$ where the fiber $X_t:=\pi^{-1}(t)$ is a smooth K3 surface for all $t\neq 0$. The central fiber $X_0$ has reduced normal crossings and all of its components are K\"ahler, moreover $\omega_X = \mathcal{O}_X$. A Kulikov model is said to be a \emph{Type II degeneration} if 
$X_0=\bigcup_{i=0}^nV_i$
is a chain of elliptic ruled surfaces with rational surfaces at each end and all double curves are smooth elliptic curves. We also say that a Type II degeneration is a \emph{Tyurin degeneration} if $X_0=V_0\cup V_1$ where both $V_0$ and $V_1$ are rational. We will refer to such an $X_0$ as a \emph{Tyurin Kulikov surface}. 
\end{defn}

\begin{notation}\label{not: restriction of line bundle on a degeneration}
Let $X\rightarrow\Delta$ be a Type II degeneration and let $L \in \operatorname{Pic}(X)$. We define $L_t:=L\big{|}_{X_t}$ for all $t \in \Delta$. In particular $L_0:=L\big{|}_{X_0}$.
\end{notation}

\begin{defn}
A \emph{$\langle 2k \rangle$-quasipolarized Tyurin degeneration} $(X\rightarrow \Delta, L)$ consists of a Tyurin degeneration $X\rightarrow \Delta$ with a line bundle $L \in \operatorname{Pic}(X)$ such that $(X_t,L_t)$ is a $\langle 2k\rangle$-quasipolarized K3 surface for $t \in \Delta^*:=\Delta\setminus\{0\}$ and $L_0^2=2k$. We also refer to $(X,L)$ as a \emph{Tyurin degeneration of K3 surfaces of degree 2$k$}. 
\end{defn}

To link Type II degenerations of K3 surfaces to $\overline{\mathcal{F}_{4}}$ we require some lattice theoretic definitions. 

\begin{defn}
Let $(L, ( \text{ },\text{ }))$ be a lattice. A vector $v \in L$ is \emph{isotropic} if $(v,v)=0$. Additionally, a sublattice $I \subset L$ is said to be \emph{isotropic} if the bilinear form $( \text{ },\text{ })\big{|}_I$ on $I$ is trivial.
\end{defn} 

\begin{defn}
Two lattices $L$ and $M$ are said to have the same \emph{genus} if $L\otimes\Z_p\simeq M\otimes\Z_p$ for all primes $p$ and $L\otimes\R\simeq M\otimes \R$. We denote by $\mathscr{G}(L)$ the set of lattices which have the same genus as $L$, up to isomorphism. We also define $\mathscr{G}(k):=\mathscr{G}(\langle -2k\rangle +E_8+ E_8)$.
\end{defn}

The following correspondences were shown by Scattone \cite{Sca87}.

\begin{theorem}
\cite[5.4.7]{Sca87} Let $I \subset \Lambda_{2k}$ be a rank 2 isotropic sublattice. Let $C_I \subset \overline{\mathcal{F}_{2k}}$ be a boundary curve, then we have the bijections
\begin{align*}
 I^\perp/I \in \mathscr{G}(k)\leftrightarrow I \text{ } \operatorname{modulo} \text{ }O(\Lambda_{2k}) \leftrightarrow C_I.
\end{align*}
\end{theorem}

It was explicitly noted by Scattone \cite[Section 6.2]{Sca87} that one can use the weight filtration given by monodromy of the degeneration to determine the isomorphism classes of $I^\perp/I\in \mathscr{G}(k)$. Indeed, the weight filtration for Type II degenerations of K3 surfaces is given by $I=(W_1)_\Z \subset I^\perp =(W_2)_\Z \subset \Lambda_{2k}$ where $W_1$ and $W_2$ are the first and second weight graded pieces of the limiting mixed Hodge structure of $X_0$ respectively. Friedman \cite[Lemma 4]{Fried84} also showed that the lattice $I^\perp/I = (W_2/W_1)_\Z$ can be related to the geometry of the central fiber using the Clemens-Schmid exact sequence in the following way. Given $X \rightarrow \Delta$ a Tyurin degeneration of K3 surfaces of degree 4, set $h:=L_0$ and define $\xi:=\mathcal{O}_X(V_0)$. Consider the lattice 
$$\mathscr{L}:= h ^\perp \subset \xi^\perp/(\Z\xi) \subset  H^2(V_0)\oplus H^2(V_1),$$
where the orthogonal complements are taken with respect to the intersection form on $H^2(V_0)\oplus H^2(V_1)$. Following Friedman \cite[5.1]{Fried84} we define the reflection $R_v(w)=w-2\frac{(v,w)}{(v,v)}v$ for $v,w\in\mathscr{L}$ and define
$$\Phi(\mathscr{L}) = \{v \in \mathscr{L} : \text{ } v \text{ is primitive, } (v,v) < 0, \text{ 
}R_v(\mathscr{L})=  \mathscr{L}\}.$$
By considering $\operatorname{Span}(\Phi(\mathscr{L}))$, one obtains a rank 17 lattice which Scattone refers to as the \textit{generalised type} of $\mathscr{L}$, and is denoted by $\mathscr{L}(W_2/W_1)_0$ in the language of Friedman \cite{Fried84}. It is worth noting that a lattice is never determined by its generalised type, however the generalised type is sufficient to distinguish the isomorphism classes of lattices in $\mathscr{G}(2)$. Scattone was able to classify all elements $\mathscr{G}(1)$ and $\mathscr{G}(2)$, and proved the following result.

\begin{theorem}\cite[Section 6.3]{Sca87}\label{thm:scattone lattices}
Let $I \subset \Lambda_{4}$ be a rank 2 isotropic sublattice. Then the generalised type of $I^\perp/I\in\mathscr{G}(2)$ is one of the following rank 17 root lattices,
\begin{align*}
    A_{11}+E_6,\qquad A_1+A_1+A_{15}, &\qquad D_8+D_8+\langle -4 \rangle, \qquad D_{12}+D_5,\qquad D_{16}+\langle -4 \rangle,\\ E_7+E_7+A_3, &\qquad E_8+E_8+\langle -4 \rangle, \qquad E_8+D_9, \qquad D_{17}.
\end{align*}
Moreover, the generalised type distinguishes the isomorphism classes of lattices in $\mathscr{G}(2)$. Therefore the Baily-Borel compactification $\overline{\mathcal{F}_4}$ has 9 boundary curves.
\end{theorem}

As the boundary curves $C_I \subset \overline{\mathcal{F}_4}$ correspond to Type II degenerations we refer to the curves $C_I$ as \emph{Type II boundary components}. To obtain the lattices of Theorem \ref{thm:scattone lattices}, in Section \ref{sect:Explicit descriptions of families} we construct explicit Tyurin degenerations of degree 4 K3 surfaces and calculate $\operatorname{Span}(\Phi(\mathscr{L}))$. We conclude this section with several definitions regarding Type II degenerations of K3 surfaces which were stated by Alexeev and Engel \cite{AE23}.

\begin{defn}
Let $L^*\in\operatorname{Pic}(X^*)$, where $L^*$ is relatively nef and big over $\Delta^*$. A relatively nef extension $L$ to a Kulikov model $X$ over $\Delta$ is called a \emph{nef model}.
\end{defn}

\begin{defn}
Let $H^*\subset X^*$ be the vanishing locus of a section of $L^*$ as above, containing no vertical components. A \emph{divisor model} is an extension $H \subset X$ to a relatively nef divisor $H\in|L|$ for which $H_0$ contains no strata of $X_0$.
\end{defn}

\begin{defn}
The \emph{stable model} of $(X^*,H^*)$ is 
$$(\overline{X},\overline{H}):=\operatorname{Proj}_\Delta\bigoplus_{n\geq0}H^0(X,\mathcal{O}(nH))$$
for some divisor model. It is unique and independent of the choice of divisor model $(X,H)$.
\end{defn}

The existence of nef models is due to Shepherd-Barron \cite{SB83}, while the existence of divisor models is due to Laza \cite[Theorem 2.11]{Laza16}.

\subsection{The Period Map for Kulikov Surfaces} For this section, we will make use of the following details. Let $X\rightarrow \Delta$ be a Kulikov model and denote singular cohomology by $H^i(-)$. We begin by recalling the definition of $d$-semistability.

\begin{theorem}\cite{Fried84}
A Tyurin Kulikov surface $X_0=V_0\cup_E V_1$ admits a smoothing to a K3 surface if and only if $N_{E/V_0}\otimes N_{E/V_1}=\mathcal{O}_E$. When this condition is satisfied, we say that $X_0$ is \emph{$d$-semistable}.
\end{theorem}

A period map for $d$-semistable Tyurin Kulikov surfaces was constructed by Carlson \cite{Carlson85}, we now restate this period map using notation of Alexeev and Engel \cite{AE23}. 

\begin{notation}
Let $X_0 = V_0\cup V_1$ be a Tyurin Kulikov surface glued along an elliptic curve $E$. We denote the image of $E$ under the embedding $E\hookrightarrow V_0$ by $E_{0,1}$ and the image of $E$ under the embedding $E \hookrightarrow V_1$ by $E_{1,0}$. We define $\xi:=(-E_{0,1},E_{1,0})\in H^2(V_0)\oplus H^2(V_1)$.
\end{notation}
\begin{defn}
The \emph{numerically Cartier} divisors on a Tyurin Kulikov surface $X_0=V_0\cup V_1$ are the elements of $\xi^\perp\subset H^2(V_0)\oplus H^2(V_1)$ where the orthogonal complement is taken with respect to the intersection pairing. We use $\widetilde{\Lambda}(X_0)$ or $\widetilde{\Lambda}$ to denote the set of numerically Cartier divisors and we set $\Lambda:=\widetilde{\Lambda}/\Z\xi$.
\end{defn}

Notice that Tyurin degenerations are normal crossings degenerations and so there is a deformation retract $c:X\times[0,1]\rightarrow X_0$ which is called the \emph{Clemens collapse}. Therefore $H^*(X_0)=H^*(X)$ and the map $c_t^*:H^*(X_0)\rightarrow H^*(X_t)$ coincides with the restriction from $X$ to $X_t$. Mayer-Vietoris for a Tyurin Kulikov surface implies that there is an exact sequence $0\rightarrow\Z^2\rightarrow H^2(X_0)\rightarrow \widetilde{\Lambda}(X_0)\rightarrow 0$ where $\Z^2$ arises from $H^1(E_{0,1})$. 

\begin{defn}
Let $I_0$ be the sublattice $\Z^2\simeq H^1(E_{0,1})\subset H^2(X_0)$ arising from Mayer-Vietoris and define $I:=c_t^*(I_0)\subset H^2(X_t)$.
\end{defn}

Notice from these definitions that $H^2(X_0)/I_0=\widetilde{\Lambda}$. To define the period map for Tyurin Kulikov surfaces, we repeat a construction of Alexeev and Engel \cite[Construction 4.3]{AE23}, originally due to Carlson \cite{Carlson85}, which determines the period point for such a surface. 

\begin{construct}\label{construct:period point}

Let $\alpha=(\alpha_0,\alpha_1)\in\widetilde{\Lambda}$ with $\alpha_0\in H^2(V_0)$ and $\alpha_1\in H^2(V_1)$. Furthermore, the $\alpha_i$ define line bundles $L_i\in \operatorname{Pic}(V_i)$ because both $V_i$ are rational elliptic surfaces. Define $$\widetilde{\psi}(\alpha):=L_0\big{|}_{E_{0,1}}\otimes L_1^{-1}\big{|}_{E_{1,0}}\in\operatorname{Pic}^0(E_{0,1})=E.$$
We say that this $\widetilde{\psi}_{X_0} \in \operatorname{Hom}(\widetilde{\Lambda},E)$ is the \emph{period point} of $X_0$, or sometimes the \emph{period morphism} of $X_0$.

\end{construct}

Furthermore, Alexeev and Engel \cite[Prop. 4.4]{AE23} show that the surface $X_0$ is smoothable if and only if the period point $\widetilde{\psi}_{X_0} \in \operatorname{Hom}(\widetilde{\Lambda},E)$ descends to a period point $\psi_{X_0}\in \operatorname{Hom}(\Lambda,E)$. In order to define a period map for a family of Tyurin Kulikov surfaces, we first need to define a marking for such a surface. Let $\Lambda_0$ denote a model for $I^\perp /I$ in $\Lambda_{K3}$. 

\begin{defn}\label{def:marking Kulikov surface}
Let $X_0$ be a $d$-semistable Tyurin Kulikov surface. A \emph{marking} consists of:
\begin{enumerate}
    \item An isometry $\sigma:\Lambda(X_0)\rightarrow \Lambda_0$ and
    \item An ordered basis $\underline{b}=(b_1,b_2)$ of $I_0\subset H^2(X_0)$.
\end{enumerate}
\end{defn}

\begin{defn}\label{def:period map family of Kulikov surfaces}
Let $(\mathcal{X}_0\rightarrow S,\sigma)$ be a family of marked $d$-semistable Tyurin Kulikov surfaces. The \emph{period map} is defined by 
$$S\rightarrow \operatorname{Hom}(\Lambda_0,\widetilde{\mathcal{E}}), \qquad s\mapsto B\circ\Psi_s\circ\sigma^{-1}.$$
Here, $\Psi_s:H^2(X_0)\rightarrow H^2(X_0,\mathcal{O})$ refers to the map arising from the exponential exact sequence and $\widetilde{\mathcal{E}}\rightarrow \C \backslash \R$ is the universal marked elliptic curve $\C/\Z \oplus \Z\tau\rightarrow \{\tau\in \C\backslash\R\}$. The map $B$ is the quotient map
$$B:H^2(X_0,\mathcal{O})\rightarrow \C/\Z\oplus\Z\tau$$
induced by the ordered basis $\underline{b}$ of $I_0$, where $b_1$ is mapped to $1\in\C$.
\end{defn}

\begin{defn}
For a rank 2 isotropic lattice $I \subset \Lambda_{K3}$ we define the sets 
$$\D = \PS\{x \in \Lambda_{K3}\otimes\C \text{ }\big{|}\text{ } x\cdot x=0,x\cdot\bar{x}>0\},\quad
\D^I:=\{ x\in\D \text{ }\big{|} \text{ } I\otimes\R\overset{\cdot x}{\rightarrow} \text{ is injective}\},  \quad \D(I):=\D^I/U_I,$$
where $U_I$ is a central $\Z$-extension of $\operatorname{Hom}(I^\perp/I,I)$. 
\end{defn}

Note first of all that $\D^I\subset\D$ is an open subset and can be considered as the domain of Hodge structures on $\Lambda_{K3}$ where $I$ also determines a mixed Hodge structure. For Type II degenerations there is also an open embedding $\D(I)\hookrightarrow A_I$ where $A_I\rightarrow I^\perp/I\otimes\widetilde{\mathcal{E}}$ is a punctured holomorphic disc bundle \cite[Prop. 4.13]{AE23}.

\begin{defn}
Define an enlargement $\D(I)\hookrightarrow\D(I)^\lambda$ to be the holomorphic disc bundle $\overline{A_I}$ extending the punctured disc bundle $A_I$.
\end{defn}

The boundary divisor of $\D(I)^\lambda$ is naturally isomorphic to $I^\perp/I\otimes \widetilde{\mathcal{E}}$ which is identified with $\operatorname{Hom}(\Lambda,\widetilde{\mathcal{E}})$ because $I^\perp/I$ is unimodular. For the remainder of Section \ref{sect:Background} we fix $h \in \Lambda_0$ a primitive vector with $h^2=4$.

\begin{defn}
An \emph{$\langle h \rangle$-quasipolarized $d$-semistable Tyurin Kulikov surface} is a pair $(X_0,j_0)$ consisting of a $d$-semistable Tyurin Kulikov surface and an embedding $j_0:\langle h \rangle \hookrightarrow \operatorname{ker}(\psi_{X_0})\subset \Lambda(X_0)$ which satisfies the following condition: For any smoothing $X\rightarrow \Delta$ with an embedding $j:\langle h \rangle \hookrightarrow \operatorname{Pic}(X)$ which induces $j_0$, the restriction to the general fiber defines a $\langle h \rangle$-quasipolarization on $X_t$. 
\end{defn}

With this definition, the period point in $\operatorname{Hom}(\Lambda,E)$ descends to a period point in $\operatorname{Hom}(\Lambda/j_0(h), E)$ which we also denote by $\psi_{X_0}$. Then the discussion of the period map above still holds, replacing $H^2(X_t)$ with $j(h)^\perp$, $\Lambda_{K3}$ with $h^\perp$, $\D^I$ with $\D^I_4:=\D^I\cap \D_{4}$, $I^\perp$ with $I^\perp_{h^\perp}$, $\Lambda$ with $\Lambda/j_0(h)$, $\Lambda_0$ with $\Lambda_0/h$. Note also that $\D_4^I=\D_4$ \cite[Proposition 4.17]{AE23}, and a mixed marked $\langle h \rangle$-quasipolarized family admits a period map to the toroidal extension 
$$\D_4(I)^\lambda:=\overline{\D_4(I)}\subset\D(I)^\lambda.$$
\subsection{Type II $\lambda$-families}
In this section we briefly discuss $\lambda$-families of Type II only, see \cite[Section 7]{AE23} for a more complete introduction, as well as the construction for the Type III case. Such $\lambda$-families can be seen as a topologically trivial family of Type II Kulikov surfaces. We first define some modifications, which are global versions of the elementary modifications of Friedman and Morrison \cite{FM81}.

\begin{prop}\cite[Prop. 7.23]{AE23}\label{AE23 Prop 7.23}
There is a family of Tyurin Kulikov surfaces $\mathcal{X}_0\rightarrow \operatorname{Hom}(\Lambda,\widetilde{\mathcal{E}})$ for which the period map is the identity.
\end{prop}
We call the family of Proposition \ref{AE23 Prop 7.23} the \emph{standard $\lambda$-family}. 

\begin{defn}
Let $\mathcal{X}\rightarrow S_\lambda:=\operatorname{Hom}(\Lambda,\widetilde{\mathcal{E}})$ be the standard $\lambda$-family. Let $B \subset S_\lambda$ be a smooth divisor and let $Z\rightarrow B$ be a smooth $\PS^1$-fibration for which the normal bundle of $Z$ restricts to $\mathcal{O}(-1)\oplus\mathcal{O}(-1)$ on every fiber. We say that the relative flop along $Z$ is:

\begin{itemize}
    \item [(GM0)] if $B$ is the closure in $S_\lambda$ of a Noether-Lefschetz divisor of K3 surfaces with a  $(-2)$-curve, and $Z$ is a family of $(-2)$-curves.
    \item [(GM1)] if $B=\Delta_\lambda$ is the discriminant, and $Z$ is a family of internal exceptional curves meeting a relative double curve $\mathcal{D}_{ij}$.
\end{itemize}

\end{defn}

\begin{defn}\label{def:lambda family}
A \emph{Type II $\lambda$-family} is a family of surfaces which arises from a series of GM0 and GM1 modifications applied to the standard $\lambda$-family.
\end{defn}

One can also consider Type III $\lambda$-families. For such families one requires GM2 modifications in addition to the GM0 and GM1 modifications described above. GM2 modifications are not necessary when discussing only Type II degenerations.

\begin{defn}
An \emph{$\langle h \rangle$-quasipolarized $\lambda$-family} is the restriction of a $\lambda$-family $\mathcal{X}\rightarrow S_\lambda$ to the Noether-Lefschetz locus $\D_{4}(I)^\lambda \cap S_\lambda\subset\D(I)^\lambda$, such that the embedding $\langle h \rangle\rightarrow \operatorname{Pic}(\mathcal{X}_t)$ induced by the marking defines an $\langle h \rangle$-quasipolarization on the generic fiber $\mathcal{X}_t$.
\end{defn}

\begin{defn}
A \emph{nef $\lambda$-family} is an $\langle h \rangle$-quasipolarized $\lambda$-family
$\mathcal{X}\rightarrow S_\lambda$ together with an extension of $L\in \langle h \rangle$ to a relatively big and nef line bundle $\mathcal{L}\rightarrow \mathcal{X}$.
 \end{defn}

We conclude this section with a brief discussion of toroidal compactifications and describe the partial compactification $\overline{\mathcal{F}_4}^{tor}$ of $\mathcal{F}_4$. The original source for this subject is \cite{AMRT75}. The construction is quite technical and many of the details can be simplified when considering only the Type II case, which we now describe.

For a boundary component $B_I \subset \overline{\mathcal{F}_4}$, there is a choice of ``admissible polyhedral cone'' $C(B_I)$ which are subsets of $U(B_I):=\{\text{center of the unipotent radical of the stabilizer of } B_I \text{ in } \operatorname{SO}_0(2,19)\}$. The group   $\operatorname{SO}_0(2,19)$ is the connected component containing the identity matrix of $$\operatorname{SO}(2,19):=\{A \in \operatorname{SL}(21,\R)\text{ }\big{|}\text{ } A^tI_{2,19}A=I_{2,19}\}$$ where $I_{2,19}$ is the diagonal matrix $\operatorname{diag}(1,1,-1,\dots,-1)$. Crucially, when $B_I$ is a Type II boundary component $\operatorname{dim}(U(B_I))=1$ and therefore the partial compactification with respect to $B_I$ requires no choices. Moreover, the partial compactification with respect to $B_I$ is obtained by introducing a smooth divisor isomorphic to $\operatorname{Hom}(\Lambda,\widetilde{\mathcal{E}})$ which parametrizes the limiting mixed Hodge structure whose weight filtration is given by $B_I$ \cite[Section 4]{Fried84}. After performing this process for one Type II boundary component we obtain $\D_4(I)^\lambda$ and we call the result of performing this process for each Type II boundary component $\overline{\mathcal{F}_4}^{tor}$. From this construction we see that $\lambda$-families are directly related to the Type II boundary components of $\overline{\mathcal{F}_4}^{tor}$.

\section{Explicit Descriptions of the Families}\label{sect:Explicit descriptions of families}

In this section, we construct explicit Type II degenerations of K3 surfaces of degree 4 and we demonstrate a correspondence between these degenerations and the lattices of Theorem \ref{thm:scattone lattices}. The goal of this section is to prove Theorem \ref{correspondence intro}.

The proof of Theorem \ref{correspondence intro} is performed in cases, and each subsection considers one degeneration. We fix some notation mentioned in Section \ref{sect:Background}. Let $(X\rightarrow \Delta,L)$ be a $\langle 4\rangle$-quasipolarized Tyurin degeneration, set $h:=L_0$ and define $\xi:=\mathcal{O}_X(V_0)$. Consider the lattice 
$$\mathscr{L}:= h ^\perp \subset\xi^\perp/(\Z\xi) \subset  H^2(V_0)\oplus H^2(V_1),$$
where the orthogonal complements are taken with respect to the intersection form on $H^2(V_0)\oplus H^2(V_1)$. We define the reflection $R_v(w)=w-2\frac{(v,w)}{(v,v)}v$ for $v,w\in\mathscr{L}$ and set
$$\Phi(\mathscr{L}) = \{v \in \mathscr{L} : \text{ } v \text{ is primitive, } (v,v) < 0, \text{ 
}R_v(\mathscr{L})=  \mathscr{L}\},$$
we will proceed by calculating the lattice $\operatorname{Span}(\Phi(\mathscr{L}))$ in each case.

We remark that in Sections \ref{sect:Two Quadrics}-\ref{sect: E8 singularity and no line}, $x,y,z,w$ are coordinates on $\PS^3$ and $t$ is the coordinate on $\Delta$. Throughout all of Section \ref{sect:Explicit descriptions of families} we make use of the following notation for $X_0=V_0\cup V_1$. If $V_0 = \operatorname{Bl}_n\PS^2$, we denote by $l$ the strict transform of a line in $\PS^2$ and $e_i$ denotes the exceptional curve over the point $p_i$. If instead $V_0=\operatorname{Bl}_n(\PS^1\times\PS^1)$, we denote the strict transform of one ruling by $s$ and denote the strict transform of the other ruling by $f$. We also denote the exceptional curve lying over the point $p_i$ by $e_i$. For $V_1$ we follow the same convention, but replacing $l,e_i,s$ and $f$ by $l',e_i',s'$ and $f'$ respectively. We also adopt the notation that $\operatorname{dP}_k$ is a del Pezzo surface of degree $k$.

\subsection{Two Quadrics Intersecting Transversally}\label{sect:Two Quadrics}
We begin with the family 
$$X = \{Q_1Q_2+tg_4=0\} \subset \mathbb{P}^3 \times \Delta,$$
where $Q_1, Q_2$ are two smooth quadrics intersecting transversally (cf. \cite[Case S-3]{Shah81}).

\begin{lemma}\label{Description two quadrics}
$X$ has 16 singularities which occur on the central fiber. After resolving these singularities, the central fiber of the degeneration is $\operatorname{Bl}_{16}(\PS^1\times\PS^1)\cup(\PS^1\times\PS^1)$.
\end{lemma}

\begin{proof}
One can check that the singularities occur along the locus $0=t=Q_1=Q_2=g_4$, which consists of 16 points. Therefore the central fiber consists of two Weil non-Cartier divisors. Blowing up the locus $0=t=Q_1$ resolves the 16 singularities inside $Q_1$ and is an isomorphism away from the singularities. 
\end{proof}

Denote $V_0 = \operatorname{Bl}_{16}(\PS^1\times\PS^1)$ and $V_1 = \PS^1\times\PS^1$ so that the central fiber of the degeneration is $X_0 = V_0\cup V_1$. Notice that a hyperplane section of $X_0$, which we denote by $h$, can be written $h=(s_1+f_1,s_2+f_2)$ where the $s_i$ and $f_i$ denote the section and fiber respectively of the two quadrics. From this we can determine the root lattice which this degeneration corresponds to.

\begin{lemma}
The lattice $\operatorname{Span}(\Phi(\mathscr{L})) = A_{15}+A_1+A_1$.
\end{lemma}

\begin{proof}
Notice first of all that we can obtain the $A_{15}$ root lattice by considering the collection
$$\Phi = \{(e_1-e_2,0), \dots, (e_{15}-e_{16},0)\}.$$
We can obtain two more simple roots by taking $\{(s_1-f_1,0)\}$ and $\{(0,s_2-f_2)\}$, these determine the two copies of $A_1$. Therefore we obtain the lattice $A_{15}+A_1+A_1$, as desired.
\end{proof}

\subsection{A Plane Intersecting a Cubic}\label{sect:Plane and Cubic}
Now we consider the degeneration 
$$X = \{xQ_3 +tg_4=0\} \subset \mathbb{P}^3 \times \Delta, $$
where $Q_3$ is a smooth cubic and $g_4$ is a suitably generic homogeneous quartic polynomial. Notice that the central fiber of this degeneration is a plane intersecting a cubic surface (cf. \cite[Case SS-5]{Shah81}).

\begin{lemma}
$X$ has 12 singularities which occur on the central fiber. After resolving the singularities, the central fiber of the degeneration is $\operatorname{Bl}_{12}\PS^2\cup \operatorname{dP}_3$.
\end{lemma}

\begin{proof}
One can check that the singularities occur on the locus $0=t=x=Q_3=g_4$, which consists of $12$ points. As in Lemma \ref{Description two quadrics}, we can blowup either the component $\{x=0\}$ or the component $\{Q_3=0\}$. Blowing up the locus $0=t=x$ gives the desired result.
\end{proof}

Let $X_0$ be the central fiber of this degeneration, and let $V_0 = \operatorname{Bl}_{12}\PS^2$ and $V_1 = \operatorname{dP}_3$ so that $X_0 = V_0 \cup V_1$. Let $h=(h_0,h_1)$ be a hyperplane section of $X_0$. Notice that a hyperplane section of $V_0$ is a line, while a hyperplane section of $V_1$ is anti-canonical. As $\operatorname{dP}_3 = \operatorname{Bl}_6\PS^2$ we can write $h=(l,3l'-\sum_{i=1}^6e_i')$.

\begin{lemma}
The lattice $\operatorname{Span}(\Phi(\mathscr{L})) = E_6+A_{11}$.
\end{lemma}

\begin{proof}
Notice first of all that the collection $\{(e_1-e_2,0),\dots,(e_{11}-e_{12},0) \} $ belongs to $\Phi(\mathscr{L})$, and these are the simple roots of an $A_{11}$ lattice. Similarly, $\{(0,e_{1}'-e_{2}'),\dots,(e_{5}'-e_{6}') \}$ belong to $\Phi(\mathscr{L})$, and we can see that $(0,l'-e_{1}'-e_{2}'-e_{3}')$ is another simple root. The set
$$ \{(0,e'_{1}-e'_{2}),\dots,(0,e'_{5}-e'_{6}), (0,l'-e'_{1}-e'_{2}-e'_{3})\}$$ 
determines the simple roots of an $E_6$ lattice, where $(0,e'_{3}-e'_{4})$ is the unique simple root which intersects 3 others. Thus, we have shown that $\operatorname{Span}(\Phi(\mathscr{L}))=E_6+A_{11}$ as desired.    
\end{proof}

\subsection{A Surface Non-Normal Along a Conic}\label{sect:non-normal along conic}
In the three cases which consider quartic surfaces which are non-normal along curves, (Sections \ref{sect:non-normal along conic}, \ref{sect:non-normal along line}, \ref{sect:non-normal along TC}) we construct subvarieties of rational scrolls. See \cite{Reid96} for more information on rational scrolls. In this case we consider a degeneration whose central fiber is a quartic surface which is non-normal along a smooth conic. Define 
$$ X = \{w^2C_2+wqL_1+q^2+t^2g_4=0\} \subset \mathbb{P}^3 \times \Delta$$
where $L_1$ and $C_2$ are generic linear and degree 2 equations respectively, $q:=x^2-yz$ is the non-singular conic belonging to the plane $w=0$. $X$ is a degeneration of degree 4 K3 surfaces, and is singular along $0=t=w=q$. The central fiber $X_0$ is a singular quartic surface which is non-normal along the smooth conic $w=q=0$.

\begin{lemma}\label{lem: blowup non-normal conic}
After blowing up the singular locus of this degeneration, the exceptional component at the central fiber is a double cover of $\mathbb{F}_2$ which is branched along a smooth member of $|2s+8f|$, where $s$ is the class corresponding to the unique $(-2)-$curve and $f$ is the class of a fiber.
\end{lemma}
\begin{proof}
We first blowup $\PS^3 \times \Delta$ along $w=q=t=0$ in the charts $U_1 =\{y\neq0\}$ and $U_2=\{z\neq0\}$. The planar conic $w=q=0$ is obtained from a degree 2 Veronese embedding so that $(x,y,z,w)=(uv,v^2,u^2,0)$. By using the techniques of rational scrolls we see that along this locus we obtain a subvariety of the rational scroll $\F(0,2,4)\simeq \F(-2,0,2)$. In $U_1$ the coordinate on the base is given by $u$ while on $U_2$ it is $v$, the coordinates of the fiber are given by $(a,b,c)$ on both charts where $a$ is of weight $-2$, $b$ is of weight 0 and $c$ is of weight 2. Along the locus we have blown up $X$ is defined as
$$\{b^2C_4+abL_2+a^2 +c^2g_8=0\} \subset \F(-2,0,2).$$
The subscripts on $C, L$, and $g$ have doubled because the conic is a degree 2 embedding of $\PS^1$. By considering the $a^2$ term in the above equation, we can see that this determines a double cover of $\F(0,2) = \F_2$ branched along the curve given by the equation
$$b^2L_2^2-4(b^2C_4+c^2g_8)=0$$
which is a smooth, generic member of the linear system $|2s+8f|$ as required.
\end{proof}

Several of the degenerations we construct arise as double covers of Hirzebruch surfaces, so we prove some statements about them. These statements are somewhat classical, but we could not find the exact statements in the literature so we reprove them.

\begin{lemma}\label{double cover hirzebruch rational}
If $\phi\colon X \rightarrow \mathbb{F}_n$ is a double cover branched in a smooth member of the linear system $|2s+2(n+2)f|$, then $X$ is a rational surface.
\end{lemma}

\begin{proof}
The linear system $|\phi^*f|$ is base point free because the linear system $|f|$ is base point free and finite surjective morphisms are flat. The general member of $|\phi^*f|$  has genus zero by the adjunction and genus formulas. Since the linear system $|f|$ is parametrised by the class of the negative section $s$, $|\phi^*f|$ is too and therefore $X$ is rational.
\end{proof}

\begin{prop}\label{double cover Hirzeburch is blowup of P2}
Let $X$ be a double cover of $\mathbb{F}_n$ which is branched in a generic smooth member of the linear system $|2s+(2n+4)f|$ for $n \geq 1$. Then $X = \operatorname{Bl}_{2n+9}\mathbb{P}^2$ where the points we blow up satisfy the relation 
$$(3n+9)q=(n+1)p_1+\dots +p_{2n+8}+p_{2n+9}$$
in the group law of an elliptic curve in $\PS^2$.
\end{prop}
\begin{proof}
Let $C\in|2s+(2n+4)f|$ be a smooth, generic curve and let $\phi:X\rightarrow \F_n$ be the double cover of $\F_n$ branched along $C$. By Lemma \ref{double cover hirzebruch rational} $X$ is a rational surface. Let $C \in |2s+(2n+4)f|$ be a smooth generic curve, we see from Hurwitz's formula that $X$ has $2n+8$ reducible fibers which arise from fibers in $\F_n$ which are tangent to $C$. Such reducible fibers are pairs of $(-1)$-curves which intersect in one point, we label the curves in each pair by $f_i$ and $f_i'$ for $i=1,\dots,2n+8$. Finally, we also see that $-K_X=\phi^*s$ by Hurwitz's formula, so $-K_X.f_i=1=-K_X.f_i'$ for all $i$.

Consider the map $\pi:X\rightarrow Y$ which contracts $f_i$ in each reducible fiber and denote $\pi(f_i)=p_i$, then $Y$ is a rational surface with $(-K_Y)^2=8$. Moreover, $|-K_Y|$ has a smooth, effective member which is $D:=\pi_*(\phi^*s)$. Therefore $Y$ is $\PS^1\times\PS^1$ or $\F_1$ or $\F_2$.

Suppose that $Y=\PS^1\times\PS^1$ and consider a section $S\in|s|$ such that $p_1\in S$. By the generic consideration on $C$, we can guarantee that $S\cap D=\{p_1,p\}$ where $p$ is not one of the $p_i$. Blowing up $p_1$ and contracting $f_1'$ defines a map to $\F_1$ and the strict transform of $S$ is the negative section of $\F_1$, intersecting the anti-canonical curve at the point $p$.

Now suppose that $Y=\F_2$ and let $s\subset\F_2$ be the negative section. We blowup $p_1$ and contract $f_1'$ to a point $p_1'$. This defines a map to $\F_1$ and the strict transform of $s$ is the negative section in $\F_1$. Notice that now the negative section intersects the anti-canonical curve in the point $p_1'$, however we can discount this possibility by the generic assumption on $C$, as in the previous case.

Therefore by the generic assumption on $C$ and choosing to blow down an $f_i'$ instead of an $f_i$ as necessary, we can assume that $Y=\F_1$. Moreover, setting $s$ to be the negative section in $\F_1$, we can assume that none of the $p_i$ belong to $s$ by the generic assumption on $C$. Therefore there is a $(-1)$-curve in $X$, which we denote by $E$, such that $E.(-K_X)=1$ and $E.\phi^*f=1$. We also see that $E.f_i'=1$ for all $i$ and $E.f_i=0$. As $X$ is a double cover of $\F_n$, there is an involution on $X$ which acts fiber-wise and so there is another $(-1)$-curve which we call $E'$. Notice that $E'.(-K_X)=1$, $E'.f_i=1$ and $E'.f_i'=0$ for all $i$. After performing the contraction to $\F_1$, the image of $E'$ is a smooth member of the linear system $|s+(n+4)f|$.

We contract $s$ to define a map to $\PS^2$, we denote the image of $E'$ by $\Gamma$ and the image of the anti-canonical curve is denoted by $D$. The curve $\Gamma$ is a curve of degree $n+4$ with an $(n+3)$-fold node and has no further singularities. Let $p$ be the $(n+3)$-fold node of $\Gamma$, then $D$ is a smooth elliptic curve which contains $p$ and intersects $\Gamma$ transversally at the points $p_1,\dots,p_{2n+8}$. There is one point of intersection between $\Gamma$ and $D$ which arises from the original intersection point of $E'\cap\phi^*s$. We denote this extra point by $p'$, then in the group law of $D$ we obtain the relation
$$(3n+12)q=(n+3)p+p_1+\dots +p_{2n+8}+p'$$
using linear equivalence.

Finally, we show that there is a relation between the points $p$ and $p'$. Recall that $p$ and $p'$ arise from the intersections $E\cap \phi^*s\subset X$ and $E'\cap\phi^*s\subset X$ respectively. In $X$ there is an involution acting fiber-wise which exchanges $E$ and $E'$, hence $p$ and $p'$ belong to the same fiber in $X$. After performing the contractions to $\F_1$, the points $p$ and $p'$ belong to the same fiber in $\F_1$. The image of this fiber after contracting $s$ is a line in $\PS^2$ which is tangent to $D$ at $p$ and intersects transversally at $p'$. Therefore we obtain the relation $3q=2p+p'$ in the group law of the elliptic curve $D$, from which we see that
$$(3n+9)q=(n+1)p+p_1+\dots +p_{2n+8},$$
as required.
\end{proof}

\begin{rk}
In future when we consider degenerations which use Proposition \ref{double cover Hirzeburch is blowup of P2}, $p_1$ will always denote the $(n+3)$-fold node of a curve of degree $n+4$. We will use $e_i$ to denote the exceptional curve obtained from blowing up $p_i$, therefore all fibers of $X$ belong to the linear system $|l-e_1|$. 
\end{rk}
The following corollary is an immediate consequence of Proposition \ref{double cover Hirzeburch is blowup of P2} and Lemma \ref{lem: blowup non-normal conic}
\begin{cor}
After blowing up the singular locus of the degeneration whose central fiber is a quartic surface which is non-normal along a smooth conic, the exceptional component of the central fiber is $\operatorname{Bl}_{13}\PS^2$.
\end{cor}

Denote this component of the central fiber by $V_0$ and denote the other component by $V_1$. It can be checked from the equations of blowing up that $V_1$ is smooth and is hence the normalisation of the quartic surface which is non-normal along a smooth conic. It is classically known \cite[Theorem 8.6.4]{Dolg12} that the normalisation of the quartic surface which is non-normal along a conic is $\operatorname{dP}_4$ and hence $V_1=\operatorname{dP}_4$.

\begin{lemma}
For this degeneration, $\operatorname{Span}(\Phi(\mathscr{L})) = D_{12}+D_5$.
\end{lemma}

\begin{proof}
Let $h=(h_0,h_1)$ be a hyperplane section of the central fiber so that $h_0$ and $h_1$ are hyperplane sections of the components $V_0$ and  $V_1$ respectively. Prior to blowing up, a hyperplane section of the central fiber intersects the double conic in two distinct points. After blowing up we obtain two fibers, so $h_0= 2(l-e_1)$. On $V_1$ a hyperplane section is just a hyperplane section of $\operatorname{dP}_4$ which is anti-canonical. In total we see that $h = (2(l-e_1), 3l'-\sum_{i=1}^5e_i')$ and as usual $\xi = (-D_0,D_1)$ for anti-canonical curves $D_i \subset V_i$. From this we see that the sets
\begin{align*}
\Phi_1 = \{(e_2-e_3,0),\dots,(e_{12}-e_{13},0)\} \quad \text{and} \quad \Phi_2 = \{(0,e_{1}'-e_{2}'),\dots,(0, e_{4}-e_{5}')\}    
\end{align*}
consist of elements belonging to $\Phi(\mathscr{L})$. We can include the simple roots $(l-e_1-e_2-e_3,0)$ and $(0,l'-e_1'-e_2'-e_3')$ to determine a generating set for the lattice $D_{12}+D_5$, where $(e_3-e_4,0)$ is the simple root connected to $3$ others in the copy of $D_{12}$ and $(0,e_3'-e_4')$ is the simple root connected to $3$ others in the copy of $D_5$.
\end{proof}

\subsection{A Surface Non-Normal Along a Line}\label{sect:non-normal along line}
Next we consider the degeneration whose central fiber is non-normal along a line. This is a slightly different description than the one obtained by Shah, where he instead considers a quartic surface which is non-normal along two skew lines \cite[SS-3]{Shah81}. We choose to work with the surface which is non-normal along one line because we would like to construct Tyurin degenerations as opposed to Type II degenerations with three (or more) components. We define the degeneration to be
$$X = \{(w^2z^2+wxg_2(y,z)+x^2h_2(y,z)+w^3g_1(y,z)+x^3h_1(y,z)+t^2g_4(w,x,y,z)=0 \} \subset \PS^3\times \Delta$$
where $h_i$ and $g_i$ are generic homogeneous polynomials of degree $i$. The central fiber of this degeneration is non-normal along $w=x=0$ and is smooth away from this locus. $X$ is singular precisely along the locus $w=x=t=0$.

\begin{lemma}\label{Skewlines blowup}
After blowing up the line $\{w=x=t=0\}$, the  exceptional component of the central fiber is a subvariety $V_0$ of the rational scroll $\F(0,0,1)$. Furthermore, $V_0=\operatorname{Bl}_{9}\PS^2$.
\end{lemma}

\begin{proof}
We blowup $\PS^3 \times \Delta$ along the line $w=x=t=0$ to obtain the rational scroll $\F(0,0,1)$ along the exceptional locus. The rational scroll has coordinates $(y,z;a,b,c)$. The coordinates on the base $\PS^1$ are $(y,z)$ and the coordinates on the $\PS^2$ fibers are given by $(a,b,c)$. Therefore the exceptional component of the central fiber is a subvariety of $\F(0,0,1)$.

The equation defining $V_0 \subset \F(0,0,1)$ is
$$a^2z^2+abg_2(y,z)+b^2h_2(y,z)+c^2g_4(y,z)=0.$$
We define a birational involution on this surface by projecting away from the section $(a,b,c)=(0,1,0)$. This section intersects $V_0$ only at the two points $(y_i,z_i;0,1,0)$ where $h_2(y_i,z_i)=0$. Blowing-up these two points and projecting away from the section $(0,1,0)$ determines a well-defined involution and is a map to $\F(0,1) = \F_1$. Furthermore, this map is a double cover of $\F_1$ with branch locus given by 
$$a^2g_2^2-4h_2(a^2z^2+c^2g_4)=0.$$
This is a smooth curve in the linear system $|2s+6f|$ in $\F_1$, and so $\operatorname{Bl}_2V_0=\operatorname{Bl}_{11}\PS^2$ by Proposition \ref{double cover Hirzeburch is blowup of P2} and hence $V_0=\operatorname{Bl}_9\PS^2$.
\end{proof}

In $\PS^3\times \PS^2\times \Delta$ the total transform of the central fiber has two components $V_0$ and $V_1$ which are given by the equations
$$0=t=w=x=a^2z^2+abg_2(y,z)+b^2h_2(y,z)+c^2g_4(y,z)$$
and
$$0=t=c=bw-ax=w^2z^2+wxg_2(y,z)+x^2h_2(y,z)+x^3h_1(y,z)$$
respectively. The intersection of these two components is the locus $\{c=0\}\subset V_0$ which is the double cover of the negative section in $\F_1$. From the proof of Proposition \ref{double cover Hirzeburch is blowup of P2} we see that this is an elliptic curve.

We now study $V_1$, which is smooth and is therefore the normalization of the original central fiber, which Urabe has previously considered \cite[III-C]{Ura86}. In particular, $V_1 = \operatorname{Bl}_9\PS^2$ and the morphism $V_1 \rightarrow \PS^3$ is defined by the linear system $|4l'-2e_1'-\sum_{i=2}^9e_i'|$. Notice that a hyperplane section of $V_1$ therefore belongs to the linear system $|4l'-2e_1'-\sum_{i=2}^9e_i'|$.

\begin{rk}
If one instead works with the quartic surface which is non-normal along two skew lines, a result of Urabe \cite[\S 1 (II-1)]{Ura86} shows that after resolving singularities we obtain a Type II degeneration with three components. The middle component is $\mathbb{P}_C(\mathcal{O}_C\oplus \mathcal{N})$, where $C$ is an elliptic curve and $\mathcal{N}$ is a degree 0 line bundle on $C$.
\end{rk}

\begin{lemma}
For this degeneration $\operatorname{Span}(\Phi(\mathscr{L}))= D_8+D_8+\langle -4 \rangle$.
\end{lemma}
\begin{proof}
We write $h=(h_0,h_1)$, by the above discussion we see that $h_1 = 4l'-2e_1'-\sum_{i=2}^9e_i'$ and so $h_1^2=4$. Prior to blowing-up, a hyperplane section will intersect the doubled line in one point and so $h_0$ is the class of a fiber after blowing up. This is represented by the class $l-e_1$. Therefore $h = (l-e_1,4l'-2e_1'-\sum_{i=2}^9e_i')$. The class $\xi$ is $(-3l+\sum_{i=1}^9e_i,3l'-\sum_{i=1}^9e_i')$.

Notice that the two sets of simple roots
\begin{align*}
\Phi_1 = \{(e_i-e_{i+1},0) : 2\leq i\leq8) \}, \qquad \Phi_2=\{ (0,e_i-e_{i+1}) : 2\leq i \leq 8)\} 
\end{align*}
both belong to $\Phi(\mathscr{L})$. There are also two simple roots $(l-e_1-e_2-e_3,0)$ and $(0,l'-e_1'-e_2'-e_3')$ which can be appended to $\Phi_1$ and $\Phi_2$ respectively. These two sets of simple roots determine two copies of the lattice $D_8$, where $(e_3-e_4,0)$ and $(0,e_3'-e_4')$ are the simple roots which intersect three others in the two copies.

We can obtain an additional generalised root $v=(-(l-e_1),2l'-\sum_{i=2}^9e_i')\in \Phi(\mathscr{L})$ which has self-intersection $-4$. One can check that $v.h=v.\xi=0$ and it is clear to see that $v$ is primitive. We show its reflection $R_v$ is an integral automorphism. Recall that $R_v$ is defined to be the map such that
$$R_v(w) = w-2\frac{(v,w)}{(v,v)}v = w+\frac{(v,w)}{2}v.$$
Therefore we show that $v.w$ is even for all $w \in \mathscr{L}$. By setting $w=(al+\sum_{i=1}^9b_ie_i, cl'+\sum_{i=1}^9d_ie_i')$, one sees that
$a+b_1+4c + 2d_1+\sum_{i=2}^9d_i=0$ and so $(v,w)=-2(a+c+d_1) \in 2\Z$. Therefore we have shown that $\operatorname{Span}(\Phi(\mathscr{L}))=D_8+D_8+\langle -4 \rangle$.
\end{proof}

\subsection{A Surface Non-Normal Along a Twisted Cubic}\label{sect:non-normal along TC}
We now consider the final degeneration which is non-normal along a curve, the central fiber is a quartic surface which is non-normal along a doubled twisted cubic. Define $F_0=xw-yz, F_1=xz-y^2, F_2=yw-z^2$ and notice that the ideal $\langle F_0,F_1,F_2\rangle \subset \C[x,y,z,w]$ defines the twisted cubic in $\PS^3$. We define the degeneration
$$X =\{AF_0^2+BF_1^2+DF_2^2+t^2g_4(x,y,z,w)=0\} \subset \mathbb{P}^3 \times \Delta,$$
where $A,B,D \in \C$ are constants. The central fiber is a quartic surface which is non-normal along a doubled twisted cubic. The singular locus of $X$ is given by the equations $0=t=F_0=F_1=F_2$ and is non-singular elsewhere. We proceed by blowing up this singular locus.

\begin{lemma}\label{exceptional piece double TC}
The exceptional divisor of the blowup of $\PS^3 \times \Delta$ along the locus $0=t=F_0=F_1=F_2$ is isomorphic to $\F(0,0,5)$.
\end{lemma}

\begin{proof}
We blowup the singular locus in the two charts $\{x\neq0\}$ and $\{w \neq0\}$ to obtain equations in two copies of $\mathbb{A}^3\times \PS^3\times\Delta$ where $(a,b,c,d)$ are the coordinates on $\PS^3$. We parameterise the twisted cubic using the usual Veronese embedding so that $(x,y,z,w) = (u^3,u^2v,uv^2,v^3)$. It can be seen from considering only the blowup relations involving $t,F_0,F_1$ and $F_2$ that the exceptional divisor is a subvariety of the rational scroll $\F(0,0,0,6)$. The coordinates on the rational scroll are $(u,v;a,b,c,d)$.

As the twisted cubic is not a complete intersection we also have the two extra equations
$$0=bz-ay-cw=by-ax-cz \subset \mathbb{A}^3\times\PS^3\times\Delta$$ which arise by computing the Rees ideal. On the rational scroll this simplifies to just the equation $bvu=au^2+cv^2 \subset \F(0,0,0,6) \simeq \F(-6,-6-6,0)$. The image of the embedding
\begin{align*}
    \mathbb{F}(-5,-5,0) &\rightarrow \mathbb{F}(-6,-6,-6,0)\\
    (u,v;\alpha,\beta,d)  &\mapsto (u,v;\alpha v,\alpha u + \beta v, \beta u, d)
\end{align*}
is precisely the locus $\{bvu=au^2+cv^2\} \subset \F(0,0,0,6)$, obtaining the result.
\end{proof}

\begin{rk}
The occurrence of $\F(0,0,5)$ is due to the normal bundle of the twisted cubic in $\mathbb{P}^3$ being $\mathcal{O}(5)\oplus \mathcal{O}(5)$ \cite[Prop. 6]{EVdV81}.
\end{rk}

Therefore the exceptional component of the central fiber after performing this blowup, denoted by $V_0$, is a subvariety of the rational scroll $\F(-5,-5,0)\simeq \F(0,0,5)$.

\begin{lemma}
Let $V_0$ be the exceptional component of the central fiber of the degeneration, then $V_0 = \operatorname{Bl}_{17}\PS^2$.
\end{lemma}

\begin{proof}
By using the explicit embedding from the proof of Lemma \ref{exceptional piece double TC} we obtain the equation
$$\alpha^2(Au^2+Dv^2)+\alpha\beta(2Auv)+\beta^2(Av^2+Bu^2)+d^2g_{12}(u,v) =0$$
in $\F(0,0,5)$. We define a birational involution by projecting away from the section $(\alpha,\beta,d)=(1,0,0)$. The section intersects the surface in two points $(u_i,v_i;1,0,0)$ where $(u_i,v_i)$ are solutions to the equation $Au^2+Dv^2=0$. After blowing up these two points we obtain a well-defined involution which is a double cover of $\F_5$ branched along the curve
$$\beta^2(2Auv)^2-4(Au^2+Dv^2)(\beta^2(Av^2+Bu^2)+d^2g_{12})=0.$$
This is a generic smooth curve in the linear system $|2s+14f| \subset \F_5$. Using Proposition \ref{double cover Hirzeburch is blowup of P2} with $n=5$ shows that $\operatorname{Bl}_2V_0=\operatorname{Bl}_{19}\PS^2$ and the result follows.
\end{proof}

The strict transform of the central fiber after blowing up the singular locus of the degeneration, denoted by $V_1$, is given by the equations 
$$0=bz-ay-cw=by-ax-cz = Ab^2+Bc^2+Da^2 \subset \mathbb{P}^3\times \mathbb{P}^2.$$
This defines a surface which is ruled over a planar conic, the fibers over each point of the conic are a copy of $\PS^1$. Additionally, for each choice of $(z,w)\in\PS^1$ one can check that we obtain another copy of $\PS^1$ which intersects each of the fibers in 1 point. Therefore we see that $V_1=\PS^1\times\PS^1$.

The intersection curve $V_0\cap V_1$ is given by the equation $\{d=0\}\subset V_0$ which is a double cover of the negative section in $\F_5$ branched in 4 points. Therefore the double curve is an elliptic curve.

\begin{lemma}
For this degeneration, $\operatorname{Span}(\Phi(\mathscr{L}))=D_{16}+\langle -4 \rangle$.
\end{lemma}

\begin{proof}
Writing $h=(h_0,h_1)$ we see that $h_0 = 3(l-e_1)$ because prior to blowing up, a hyperplane section will intersect the twisted cubic in 3 points. Therefore after blowing-up we will obtain 3 fibers. On the other hand, $h_1=s+2f$ because $h_1^2=4$ and must intersect the double curve in 6 points. In this case $\xi=(-3l+\sum_{i=1}^{17}e_i,2s+2f)$.

Notice that the collection 
$$\Phi_1 = \{(e_i-e_{i+1},0) : 2\leq i\leq 16 \}$$
is contained in $\Phi(\mathscr{L})$ and is a collection of 15 simple roots. We can further append the simple root $(l-e_1-e_2-e_3,0)$ to $\Phi_1$ to obtain a basis of simple roots for the $D_{16}$ lattice.

We can obtain a generalised root $v=(l-e_1,2f-s)$ which has self-intersection $-4$. One can check that for $w = (al+\sum_{i=1}^{17}b_ie_i,cs+df)\in \mathscr{L}$ we have $(v,w) = -2(a+b_1+d)\in2\Z$ and so the reflection $R_v$ is an integral automorphism of $\mathscr{L}$, concluding the proof.
\end{proof}

\subsection{A Surface with One $\widetilde{E_7}$ singularity}\label{sect: E7 singularity}
The next three degenerations we consider in Section \ref{sect:Explicit descriptions of families} have $\widetilde{E}_7$ and $\widetilde{E}_8$ singularities. We will use weighted blowups to work with these degenerations, see \cite[6.38]{KSC04} for more information on weighted blowups. The first degeneration we construct has a singular fiber which has one $\widetilde{E}_7$ singularity (cf.\cite[SS-2.1]{Shah81}). Shah actually considers quartic surfaces with two $\widetilde{E}_7$ singularities, but we can alter the equations provided by Shah to obtain just one such singularity. We define the degeneration
$$X = \{w^2z^2+2wzh_2(x,y)+h_4(x,y) +z^4 +t^4g_4(x,y,z,w)=0\} \subset \mathbb{P}^3 \times \Delta$$
where $h_i$ are homogeneous polynomials of degree $i$ in the variables $x$ and $y$, and $g_4$ is homogeneous of degree 4 in variables $x,y,z,w$. $X$ has one singularity at $0=t=x=y=z$. The central fiber has one singularity which is of type $\widetilde{E}_7$ at $x=y=z=0$.

\begin{lemma}
After performing a weighted blowup at the point $x=y=z=t=0$ with corresponding weights (1,1,2,1) the total transform of the central fiber has two components $V_0 \cup V_1$. Furthermore, the exceptional component of the central fiber is $V_1 = \operatorname{dP}_2$.
\end{lemma}

\begin{proof}
We blow-up in the chart $\{w\neq0\}$ to obtain two surfaces $V_0$ and $V_1$, where $V_1$ is defined by the equations
$$0=t=x=y=z=d^2+2dh_2(a,b)+h_4(a,b) +c^4k \subset\mathbb{A}^4 \times \mathbb{P}(1,1,1,2)$$
where $a,b,c$ are the variables of weight $1$ in $\PS(1,1,1,2)$, $d$ is the weight $2$ variable and $k$ is the coefficient of $w^4$ in $g_4$. These equations define a double cover of $\PS^2$ branched along the quartic curve $\{h_2^2(a,b)-h_4(a,b)-kc^4=0\}$. Therefore $V_1 = \operatorname{dP}_2$. The other component, which we denote by $V_0$, is defined by the equations
$$0=t=c=ay-bx=a^2z-dx^2=b^2z-dy^2=d^2+2dh_2(a,b)+h_4(a,b)+d^2z^2\subset \mathbb{A}^4\times \PS(1,1,1,2)$$
which is smooth. Furthermore the intersection $V_0 \cap V_1$ is given by the equation
$$d^2+2dh_2(a,b)+h_4(a,b)=0 \subset \PS(1,1,2)$$
which is a double cover of $\PS^1$ branched in 4 points, so is an elliptic curve.
\end{proof}
We now consider the surface $V_0$ in the central fiber.

\begin{lemma}\label{lem: strict transform central fiber E7}
In this degeneration, $V_0 = \operatorname{Bl}_4\operatorname{dP}_2$.
\end{lemma}

\begin{proof}
We perform an ordinary blowup of $X$ at the point $0=t=x=y=z$ and consider the non-exceptional locus of this blowup. Notice that the normalisation of this surface is $V_0$ because $V_0$ is smooth and the fact that the normalisation of reduced schemes is unique. We rewrite the equations so that $h_2(x,y) = xf_1(x,y)+Ky^2$ and $h_4(x,y) = x^2g_2(x,y)+y^3g_1(x,y)$ and we obtain the system of equations
$$0=ay-bx=az-cx=bz-cy=c^2+2c\left(af_1(x,y)+Kby\right)+a^2g_2(x,y)+b^2yg_1(x,y)+c^2z^2 \subset \A^3\times\PS^2$$
where $a,b,c$ are the coordinates on $\PS^2$. Projectivising the final equation determines a subvariety of $\PS^3 \times \PS^2$, and we obtain the equations
$$0=ay-bx=az-cx=bz-cy=w^2c^2+2wc\left(af_1(x,y)+Ky^2\right)+a^2g_2(x,y)+b^3g_1(x,y)+c^2z^2 \subset \PS^3\times\PS^2$$
where $x,y,z,w$ are the coordinates on $\PS^3$. We call this variety $V_0'$ and one can check that $V_0'$ is non-normal along the line $0=c=x=y=z$. Notice that the locus $c=0$ consists of the non-normal line and the 4 lines defined by the solutions of $h_4(a,b)=0$. We denote these lines by $L_1,\dots,L_4$.

From the first three equations, we see that picking a point $(a,b,c)\in\PS^2$ uniquely determines the choices for $x,y$ and $z$ in $V_0'$. Therefore the final equation determines a (birational) double cover of $\PS^2$ which is branched along the curve given by the equation $c^2\left((h_2(a,b))^2-h_4(a,b)-c^4\right)=0$, which consists of a smooth quartic curve and a repeated line. This map to $\PS^2$ also contracts the lines $L_1\dots,L_4$ to points on the line $\{c=0\}\subset\PS^2$. As $V_0$ is the normalization of $V_0'$ there is a double covering map $V_0\rightarrow\PS^2$ which is branched along a smooth quartic and contracts 4 lines, and hence $V_0=\operatorname{Bl}_4\operatorname{dP}_2$.
\end{proof}

\begin{lemma}\label{lemma: root lattice E7E7A3}
For this degeneration, $\operatorname{Span}(\Phi(\mathscr{L}))=E_7+E_7+A_3$.
\end{lemma}

\begin{proof}
Setting $h=(h_0,h_1)$ where $h$ arises as a hyperplane section of the original central fiber we see that $h_1=0$ because a general hyperplane will not pass through the singular point.

Furthermore a general hyperplane section will intersect each of the lines $L_1,\dots,L_4$ defined in Lemma \ref{lem: strict transform central fiber E7} in one point. After contracting these lines, the hyperplane section has self-intersection 8, and therefore its image under the double covering map is a conic in $\PS^2$. From this one can see that $h_0 \in |6l-2\sum_{i=1}^7e_i-\sum_{i=8}^{11}e_i|$. It is then quick to see that the collections
\begin{align*}
\Phi_1 = \{(e_i-e_{i+1},0) &\mid 1\leq i\leq 6\}, \quad  \Phi_2 = \{(0,e_i'-e_{i+1}') \mid 1\leq i\leq 6\},\\
&\Phi_3 = \{(e_i-e_{i+1},0) \mid 8\leq i\leq 10\}
\end{align*}
are contained in $\Phi(\mathscr{L})$. The collection $\Phi_3$ consists of the simple roots for the $A_3$ lattice. One can append the simple roots $(l-e_1-e_2-e_3,0)$ and $(0,l'-e_1'-e_2'-e_3')$ to $\Phi_1$ and $\Phi_2$ respectively to obtain the simple roots for two copies of the $E_7$ lattice.
\end{proof}

\subsection{A Surface with an $\widetilde{E_8}$ singularity and a line passing through the singularity}\label{sect: E8 singularity and a line}
Our next degeneration has a central fiber which has an $\widetilde{E}_8$ singularity and has a line passing through the singular point. This is similar to Shah's case \cite[SS-1.1]{Shah81} which has two $\widetilde{E}_8$ singularities, both of which have a line passing through them . The equation of this degeneration is given by
$$X = \{ w^2z^2+wy^3+Axyzw+Bx^2y^2+Cx^3z+Dz^4+t^6g_4(x,y,z,w)=0\}\subset \PS^3 \times \Delta$$
where $A,B,C,D$ are constants and $g_4$ is a generic homogeneous polynomial of degree 4. The central fiber has one singular point of type $\widetilde{E}_8$ at the point $0=x=y=z$. $X$ also has only one singular point at $0=t=x=y=z$.

\begin{lemma}
After performing a weighted blowup at the point $x=y=z=t=0$ with corresponding weights (1,2,3,1) the total transform of the central fiber has two components $V_0 \cup V_1$ where $V_0 = \operatorname{Bl}_{10}\PS^2$ and $V_1=\operatorname{dP}_{1}$. Furthermore, $V_0 \cap V_1$ is an elliptic curve.
\end{lemma}

\begin{proof}
After performing this weighted blowup, $V_1$ is given by the equations
$$0=t=x=y=z=\delta^2 +\gamma ^3 +A\alpha\gamma\delta +B\alpha^2\gamma^2+C\alpha^3\delta+\beta^6 k  \subset \mathbb{A}^4 \times \mathbb{P}(1,1,2,3)$$
where $\alpha,\beta,\gamma,\delta$ have weights $1,1,2,3$ respectively and $k$ is the coefficient of the $w^4$ term of $g_4$. This determines a sextic in $\PS(1,1,2,3)$, hence $V_1 = \operatorname{dP}_1$. The surface $V_0$ is given by the equations
\begin{align*}
0=t=\beta=\alpha^3z-\delta x=\alpha^2y&-\gamma x=\gamma^3z^2-\delta^2y^3\\
&=\delta^2 +\gamma ^3 +A\alpha\gamma\delta +B\alpha^2\gamma^2+C\alpha^3\delta+D\delta^2z^2+\beta^6 k \subset \A^4\times\PS(1,1,2,3).
\end{align*}
One can check that $V_0$ is smooth and is therefore the normalisation of the central fiber. Work of Urabe \cite[Proposition 1.5, p.1192]{Ura84} as well as Laza and O'Grady \cite[Remark 4.2]{LOG19} then shows that $V_1=\operatorname{Bl}_{10}\PS^2$. Furthermore, $V_0 \cap V_1$ is a sextic curve in $\PS(1,2,3)$ which is an elliptic curve.
\end{proof}

\begin{lemma}\label{lemma:root lattice E8E8}
For this degeneration, the lattice $\operatorname{Span}(\Phi(\mathscr{L}))=E_8+E_8+\langle -4 \rangle$.
\end{lemma}

\begin{proof}
A general hyperplane section of the central fiber will not intersect the singular point and by \cite[Proposition 4.3]{Ura84} we see that $h=(9l-3\sum_{i=1}^8e_i-2e_9-e_{10},0)$. We have the two collections of simple roots
$$\Phi_1 = \{(e_i-e_{i+1},0) \mid 1\leq i\leq 7\}, \qquad  \quad \Phi_2 = \{(0,e_i'-e_{i+1}') \mid 1\leq i\leq 7\}.$$ 
We can append the simple roots $(l-e_1-e_2-e_3,0)$ and $(0,l'-e_1'-e_2'-e_3')$ to $\Phi_1$ and $\Phi_2$ respectively to obtain two copies of the $E_8$ lattice. There is also the class $v=(e_9-2e_{10},-3l'+\sum_{i=1}^8e_i')$ which has self-intersection $-4$. By considering an arbitrary $w \in \mathscr{L}$, one can check that $(v,w)\in 2\Z$ and hence the reflection $R_v$ is an integral automorphism. Therefore $v$ is a generalised root, completing the proof.
\end{proof}

\subsection{A Surface With One $\widetilde{E_8}$ singularity}\label{sect: E8 singularity and no line}
We consider a degeneration whose central fiber is a quartic surface with one $\widetilde{E}_8$ singularity which does not have a line passing through the singularity. The equation for such a singular surface is given by Shah \cite[S-4.1]{Shah81} and the degeneration is given by the equation
\begin{align*}
   X = \Bigl\{0= \left(wz+x^2+ay^2\right)^2 &+y^3(w+f_1(x,y,z))+y^2f_2(x,z)\\&+yzg_2(x,z)+z^3g_1(x,z)+t^6g_4(x,y,z,w) \Bigl\} \subset \mathbb{P}^3 \times \Delta
\end{align*}
where $f_i$ and $g_i$ are generic homogeneous equations of degree $i$.

\begin{lemma}\label{blow 1 E8}
After performing a weighted blowup at the point $x=y=z=t=0$ with corresponding weights (1,2,3,1) the total transform of the central fiber has two components $V_0 \cup V_1$ where $V_0 = \operatorname{Bl}_{10}\PS^2$ and $V_1=\operatorname{dP}_{1}$. Furthermore, $V_0\cap V_1$ is a smooth elliptic curve.
\end{lemma}

\begin{proof}
After performing this weighted blowup, $V_1$ is given by the equations
$$0=t=x=y=z=\delta^2+\gamma^3+a_6\alpha^2\gamma^2+a_3\alpha^4\gamma+\beta^6k \subset \A^4\times\PS(1,1,2,3)$$
where $\alpha,\beta,\gamma,\delta$ have weights $1,1,2,3$ respectively. The coefficients $a_6, a_3$ are arbitrary and $k$ is the coefficient of $w^4$ in $g_4$. This form is obtained using Shah's initial form \cite[p. 281]{Shah81}. One can see that this equation determines a sextic in $\PS(1,1,2,3)$ so $V_1=\operatorname{dP}_1$.

One can check that $V_0$ is smooth and hence $V_0=\operatorname{Bl}_{10}\PS^2$ by work of Urabe \cite[Proposition 1.5, p.1192]{Ura84} as well as Laza and O'Grady \cite[Remark 4.2]{LOG19}. It is also easy to check that $V_0\cap V_1$ is an elliptic curve.
\end{proof}

\begin{lemma}\label{lemma:root lattice E8D9}
For this degeneration, the lattice $\operatorname{Span}(\Phi(\mathscr{L}))=E_8+D_9$.
\end{lemma}

\begin{proof}
Let $h$ be a hyperplane section of the central fiber, a generic choice will not intersect the singular point, so $h=(h_0,0)$. Urabe \cite[Proposition 4.3]{Ura84} and Laza and O'Grady \cite[Remark 4.2]{LOG19} show that $h_0 = 7l-3e_1-2\sum_{i=2}^{10}e_i$. The lattice $\Phi(\mathscr{L})$ contains the two collections 
$$\Phi_1 = \{(e_i-e_{i+1},0) : 2 \leq i \leq 9) \}, \qquad \Phi_2 = \{(0,e_i'-e_{i+1}') : 1\leq i\leq7 \}.$$
One can append the simple root $(l-e_1-e_2-e_3,0)$ to $\Phi_1$ to determine an $D_9$ lattice. One can also append the simple root $(0,l'-e_1'-e_2'-e_3')$ to $\Phi_2$ to determine a $E_8$ lattice completing the proof.
\end{proof}

\subsection{A Hyperelliptic Degeneration}\label{sect:hyperelliptic degen}
In this section we consider the degeneration which corresponds to the $D_{17}$ lattice. Before we construct this degeneration, we remark that the general fiber is a hyperelliptic K3, as opposed to all the previous constructions where the general fiber is a quartic surface in $\PS^3$. Laza and O'Grady \cite[7.3]{LOG19} note that the $D_{17}$ stratum in $\overline{\mathcal{F}_4}$ is not visible as a Type II quartic surface in the sense of Shah \cite[Theorem 2.4]{Shah81}. Instead it is seen as a Type IV surface and is contained in the $E_{12}$ stratum, see \cite[Theorem 2.4, S-4.3.1]{Shah81} for the equation of such a surface. We note that this degeneration was inspired by Shah's family of type 3--18c \cite[p. 306]{Shah81} and by the Type II component $D_{16}$ considered by Laza and O'Grady \cite{LOG21}. In \cite{LOG21}, the authors also remark that this Type II component is not a GIT quotient of $(4,4)$ curves in $\PS^1\times\PS^1$.

We define the degeneration $X\rightarrow \Delta$ to be the subvariety of $\PS(1,1,1,1,2)\times \Delta$ given by the equations
\begin{align*}
w^2&=\left(Ax_3+x_1-\frac{k}{2}x_2\right)(x_0x_2^2-2x_1x_2x_3+x_3^3)+t(F_0+F_1+F_2)f_2(x_i)+t^2g_4(x_i),\\
0&=x_1^2-x_0x_3+t^2Q_2(x_i),
\end{align*}
where $x_i$ and $w$ are the weight 1 and 2 variables of $\PS(1,1,1,1,2)$ respectively. $Q_2$ and $f_2$ are generic homogeneous polynomials of degree 2 and $g_4$ is a generic homogeneous polynomial of degree 4. Furthermore, we define the polynomials $F_0,F_1,F_2$ to be
$$F_0=x_1x_2-x_3^2, \quad F_1 = x_0x_2-x_1x_3,\quad F_2 = x_1^2-x_0x_3$$
whose common vanishing in $\PS^3$ define a twisted cubic. One can see that for $t \neq 0$ we obtain a $(2,4)$ complete intersection in $\PS(1,1,1,1,2)$ where the degree 2 equation does not contain the weight two variable. This is a formulation of hyperelliptic K3s of degree 4 \cite[Example 2]{HarThom15}. The central fiber is the double cover of the quadric cone branched along the desired curve. The singular locus of $X$ is $0=t=w=F_0=F_1=F_2$, which is the twisted cubic. We proceed by blowing up the singular locus.

\begin{lemma}\label{blowup D17 total space}
The exceptional divisor of the blowup of $\PS(1,1,1,1,2) \times \Delta$ along the locus $0=t=w=F_0=F_1=F_2$ is isomorphic to $\F(-6,0,-5)$.
\end{lemma}

\begin{proof}
The proof is similar to that of Lemma \ref{exceptional piece double TC}, however the twisted cubic is now parameterized by $(x_0,x_1,x_2,x_3)=(u^3,u^2v,v^3,uv^2)$ and so we blowup on the two charts $\{x_0\neq 0\}$ and $\{ x_2\neq 0\}$. By considering the blowup relations involving only $t,w,F_0,F_1,F_2$ we see that the exceptional divisor is a subvariety of the rational scroll $\F(-6,0,-6,-6,-6)$. The coordinates of this rational scroll are $(u,v;a,b,c,d,e)$. The projective coordinates $a,b,c,d,e$ are obtained from performing the blowup, and the relations obtained by blowing up are the $2\times2$ minors of
$$
\begin{pmatrix}
    w & t & F_0 & F_1 & F_2\\
    a & b & c & d & e\\
\end{pmatrix}$$
in addition to the syzygies $0=cx_1-dx_3-ex_2=cx_0-dx_1-ex_3$. From this we can see that the exceptional divisor is contained in the locus $\{e=0\} = \F(-6,0,-6,-6)\subset\F(-6,0,-6,-6,-6)$. We then use the syzygies to obtain the equation $cu=dv$, defining $\alpha$ so that $c = \alpha v$ and $d=\alpha u$ determines an embedding 
\begin{align*}
    \F(-6,0,-5) &\rightarrow \F(-6,0,-6,-6)\\
    (u,v;a,b,\alpha) &\mapsto (u,v;a,b,\alpha v,\alpha u),
\end{align*}
completing the proof.
\end{proof}

\begin{lemma}\label{lemma: blow up hyperelliptic case}
After blowing up the singular locus, the  exceptional component of the central fiber is a subvariety $V_0$ of $\F(-6,0,-5)$. Furthermore $\operatorname{Bl}_3(V_0)$ is a double cover of $\mathbb{F}_6$, branched in a smooth member of the linear system $|2s+16f|$ and $V_0=\operatorname{Bl}_{18}\PS^2$.
\end{lemma}

\begin{proof}
The first part is a direct consequence of Lemma \ref{blowup D17 total space}. We define $L_3 = \left(Auv^2+u^2v-\frac{k}{2}v^3\right)$, one can check that we obtain the subvariety
$$a^2u =L_3\alpha^2+\alpha bu(u+v)f_6+b^2(ug_{12}+v^4Q_6L_3)$$
where $f_6, Q_6, g_{12}$ arise from $f_2,Q_2,g_{4}$ respectively after performing the substitution $(x_0,x_1,x_2,x_3)=(u^3,u^2v,v^3,uv^2)$. After blowing up the three points $(u_i,v_i;0,0,1)$ where $L_3(u_i,v_i)=0$, we project away from the section $(a,b,\alpha)=(0,0,1)$ to obtain a double cover of $\{\alpha = \frac{-bu(u+v)f_6}{2L_3} \} =\F_6$. It is then easy to check that the branch curve is a smooth curve in the linear system $|2s+16f|$. Applying Proposition \ref{double cover Hirzeburch is blowup of P2} with $n=6$ gives the desired result.
\end{proof}

\begin{lemma}
Let $V_1$ be the strict transform of the central fiber, then $V_1 = \PS^2$.
\end{lemma}

\begin{proof}
We first perform the change of variables $y=\left(Ax_3+x_1-\frac{k}{2}x_2\right)$ and recall that on this component, $b=0$. The relations obtained from blowing up which use only the $x_i$ and $y$ define a double cover of the quadric cone branched in a conic and in the vertex of the cone. Therefore these equations define a copy of $\PS^2$. The equations involving $w$ show that for each choice of $x_i, y, a, c, d$ there is a unique choice for $w$, so $V_1 = \PS^2$.
\end{proof}

\begin{lemma}
For this degeneration, the lattice $\operatorname{Span}(\Phi(\mathscr{L}))=D_{17}$.
\end{lemma}

\begin{proof}
Setting $h=(h_0,h_1)$, we see that $h_0$ consists of three fibers of $V_0$ so $h_0 = 3(l-e_1)$. Therefore $h_1^2=4$ so $h_1=2l'$. One can see that the collection 
$$\Phi = \{(e_i-e_{i+1},0) \mid i=2,3,\dots,17\}$$
is contained in $\Phi(\mathscr{L})$. One can append the simple root $(l-e_1-e_2-e_3,0)$ which is also contained in $\Phi(\mathscr{L})$ and determines a $D_{17}$ lattice.
\end{proof}

This concludes the proof of Theorem \ref{correspondence intro}. 

\section{Period Domains}\label{sect:period domains}
In Section \ref{sect:Explicit descriptions of families} we explicitly constructed Type II degenerations which correspond to the 9 Type II boundary components of $\overline{\mathcal{F}_4}$. In this section we show that each Type II degeneration we constructed in Section \ref{sect:Explicit descriptions of families} is distinguished by a relation on the points we blewup. As motivation we state a proposition of Alexeev and Engel, suitably modified for our purposes.

\begin{prop}\cite[Proposition 7.33]{AE23}\label{prop: nef model extension}
Given a nef model $L\rightarrow X$ of a Type II $\langle h \rangle$-polarized Kulikov model, $X\rightarrow \Delta$, there is a nef $\lambda$-family $\mathcal{L}\rightarrow \mathcal{X}\rightarrow S_\lambda$ extending it.
\end{prop}

By the definition of $\lambda$-families, this means that there is a family of $\langle h \rangle$-polarized Kulikov models $\mathcal{X}_0\rightarrow \operatorname{Hom}(\Lambda/h, \widetilde{\mathcal{E}})$ for which the period map is the identity.

\begin{prop}\label{prop:relations}
For each of the explicitly constructed Tyurin degenerations of K3 surfaces of degree 4 constructed in Theorem \ref{correspondence intro}, the period map for families of Kulikov surfaces imposes an additional relation on the points one can blow up in Construction \ref{Tyurin constructions}. This relation is given in the final column of Table \ref{Complete Description Table}. Furthermore these relations can be obtained geometrically.
\end{prop}

In Table \ref{Complete Description Table}, the Tyurin degenerations for $E_8+E_8+\langle -4 \rangle$ and $E_6+A_{11}$ constructed in Section \ref{sect:Explicit descriptions of families} are the models given in rows 2 and 5 respectively. Notice that Proposition \ref{prop:relations} and Table \ref{Complete Description Table} describe nef $\lambda$-families which extend the nef models we constructed in Section \ref{sect:Explicit descriptions of families}. The remainder of this section is dedicated to the proof of Proposition \ref{prop:relations}. We address each degeneration in turn, referring to each case by their corresponding lattice. From now on we adopt the conventions established in Construction \ref{Tyurin constructions}.

\begin{lemma}
In the $A_{11}+E_6$ case (Section \ref{sect:Plane and Cubic}), the period map imposes the condition $12q = \sum_{i=1}^{12}p_i$ where $p_i$ are the points we blow up to obtain $\operatorname{Bl}_{12}\PS^2$. This relation is also determined by the geometry of the degeneration.
\end{lemma}

\begin{proof}
Recall that $h=(l,3l'-\sum_{i=1}^6e_i')$ and $\xi=(-3l+\sum_{i=1}^{12}e_i,3l'-\sum_{i=1}^6e_i')$. Since both the polarization $h$ and double curve $\xi$ are Cartier divisors, we see that $$3q-9q'+\sum_{i=1}^6p_i' = 0, \qquad 9q-\sum_{i=1}^{12}p_i+9q' -\sum_{i=1}^6p_i' = 0$$
and therefore $12q=\sum_{i=1}^{12}p_i$. By the construction of our degeneration in Section \ref{sect:Plane and Cubic}, the 12 points $p_i$ lie on the intersection of a planar cubic and planar quartic, therefore they satisfy the relation $12q=\sum_{i=1}^{12}p_i$.
\end{proof}

\begin{lemma}
In the $A_1+A_1+A_{15}$ case (Section \ref{sect:Two Quadrics}), the period map imposes the condition $16q=\sum_{i=1}^{16}p_i$ where the $p_i$ are the points we blow up to obtain $\operatorname{Bl}_{16}(\PS^1\times\PS^1)$. This relation is also determined by the geometry of the degeneration.
\end{lemma}

\begin{proof}
Recall that $h=(s_1+f_1,s_2+f_2)$ and $\xi=(-2s_1-2f_1+\sum_{i=1}^{16}p_i,2s_2+2f_2)$. Applying the period map to $h$ and $\xi$ we obtain the relations
$$4q-4q'=0,\qquad 8q+8q'-\sum_{i=1}^{16}p_i=0$$
from which we obtain the desired relation. Four points $p_1,p_2,p_3,p_4$ sum to zero in the group law on elliptic curves in $\PS^1\times\PS^1 \subset\PS^3$  if and only if they are coplanar \cite{Sil09}. In the construction of Section \ref{sect:Two Quadrics} we blowup 16 points which lie on the intersection of a quadric and a quartic surface. Using linear equivalence we see that $16q=\sum_{i=1}^{16}p_i$.
\end{proof}

In the cases $D_8+D_8+\langle -4 \rangle,$ $D_{12}+D_5,$ $D_{16}+\langle -4 \rangle$, and $D_{17}$ one of the components is realised as a birational double cover of $\F_n$ for some $n$. By Proposition \ref{double cover Hirzeburch is blowup of P2} these double covers can also be obtained by blowing up points in $\PS^2$ which are subject to a relation.

\begin{lemma}
In the $D_{12}+D_5$ case (Section \ref{sect:non-normal along conic}) the period map imposes the condition $15q = 3p_1+\sum_{i=2}^{13}p_i$ on the points we blowup to obtain $\operatorname{Bl}_{13}\PS^2$. This relation is also determined by the geometry of the degeneration.
\end{lemma}

\begin{proof}
Recall that $h = (2(l-e_1), 3l'-\sum_{i=1}^5e_i')$ and $\xi=(-3l+\sum_{i=1}^{13}e_i,3l'-\sum_{i=1}^5e_i')$ Applying the period map to $h$ and $\xi$ one obtains the relations
$$6q-2p_1-9q'+\sum_{i=1}^5p_i'=0, \qquad 9q+9q'-\sum_{i=1}^{13}p_i-\sum_{i=1}^{5}p_i'=0$$
from which one obtains the desired relation. By applying Proposition \ref{double cover Hirzeburch is blowup of P2} with the case $n=2$, one is able to see the relation geometrically too.
\end{proof}

In the remaining cases where we obtained a birational double cover of a Hirzebruch surface we had to blow up points and so we cannot apply Proposition \ref{double cover Hirzeburch is blowup of P2} directly.

\begin{lemma}
In the $D_{8}+D_8+\langle -4 \rangle$ case (Section \ref{sect:non-normal along line}) the period map imposes the condition $12q+p_1'=3q'+2p_1+\sum_{i=2}^9p_i$ on the points we blowup. This relation is also determined by the geometry of the degeneration.
\end{lemma}

\begin{proof}
Recall that $h = (l-e_1,4l'-2e_1'-\sum_{i=2}^9e_i')$ and $\xi= (-3l+\sum_{i=1}^9e_i,3l'-\sum_{i=1}^9e_i')$. After applying the period map to $h$ and $\xi$ one obtains the relations
$$3q-p_1 -12q' + 2p_1'+\sum_{i=2}^9p_i'=0,\qquad
    9q-\sum_{i=1}^9p_i+9q'-\sum_{i=1}^9p_i'=0$$
from which one derives the desired relation. This relation can be seen geometrically too. Recall that in Section \ref{sect:non-normal along line} we blewup the two points $(y_i,z_i;0,1,0)$ in $V_0$ to obtain a double cover of $\F_1$. Denote these points by $r_1$ and $r_2$. By Proposition \ref{double cover Hirzeburch is blowup of P2}, we see that $12q=2p_1+p_2+\dots+p_9+r_1+r_2$. 

On $V_1$ the points $r_1, r_2$ belong to the locus defined by $C := \{a=w=0\}$ which is a reducible member of the linear system $|4l'-2e_1'-\sum_{i=2}^9e_i'|$ in $V_1$. We write $C=C_1+C_2$ where $C_i$ are the two connected components of $C$. Note in particular that $C$ contains the double curve as a component, say $C_1$, and so $C_2 \in |l'-e_1'|$. Furthermore $C_1 \cap C_2$ consists of the two points $r_1,r_2$. From this we see that $3q'=p_1'+r_1+r_2$ as $p_1',r_1,$ and $r_2$ are collinear. To obtain the desired relation, we substitute this into the relation $12q=2p_1+\sum_{i=2}^9p_i+r_1+r_2$.
\end{proof}

\begin{rk}
When treating this case as having two non-normal lines which are skew, one can also obtain this relation. In that approach the two rational surfaces $V_0$ and $V_2$ are birational double covers of $\F_1$. Furthermore, one can show that the points $r_1$ and $r_2$ belong to different fibers of $V_1$. However, the points of intersection of those fibers with $V_2$ belong to a common fiber of $V_2$, from which one obtains the desired relation.
\end{rk}

\begin{lemma}
In the $D_{16}+\langle -4 \rangle$ case (Section \ref{sect:non-normal along TC}) the period map imposes the condition $63q=15p_1+3\sum_{i=2}^{17}p_i$ on the points we blowup to obtain $\operatorname{Bl}_{17}\PS^2$. This relation is also determined by the geometry of the degeneration.
\end{lemma}

\begin{proof}
We first recall some details for the group law of an elliptic curve which is embedded in $\PS^1\times\PS^1$. Let $E\subset\PS^2$ be an elliptic curve with $q$ the identity of the group law. Let $p_s$ be a point on $E$ satisfying $4p_s=4q$, note that there are only finitely many such points. Consider the line passing through $q$ and $p_s$, one can see that this residual point which we call $p_f$ satisfies $4p_f=4q$. Notice also that $p_s+p_f=2q$. Blowing up the points $p_s$ and $p_f$, then contracting the line joining them defines a map to $\PS^1\times\PS^1$. From this we obtain a fiber passing through the points $q$ and $p_s$ and a section passing through the points $q$ and $p_f$. Moreover, the points $p_s$ and $p_f$ still satisfy $4p_s=4p_f=4q$. Finally, in $\PS^1\times\PS^1$, four points $p_1,\dots,p_4$ satisfy $p_1+\dots+p_4=4q$ if and only if they lie on a curve belonging to the linear system $|s+f|$.

Recall in this case that $h=(3(l-e_1),s+2f)$ and $\xi=(-3l+\sum_{i=1}^{17}e_i,2s+2f)$. Applying the period map to $h$ and $\xi$ one obtains
$$9q-3p_1-5q'-p_f=0, \qquad 9q+8q'-\sum_{i=1}^{17}p_i=0$$
from which the required relation can be deduced using the discussion above. To determine this relation geometrically, we first recall some notation of Section \ref{sect:non-normal along TC}. In this case, $V_0=\operatorname{Bl}_{17}\PS^2$ is a subvariety of the rational scroll $\F(0,0,5)$ with coordinates $(u,v;\alpha,\beta,d)$. There is also an embedding $\F(0,0,5)\hookrightarrow\F(0,0,0,6)$ which is defined by $(u,v;\alpha,\beta,d)\mapsto (u,v;\alpha v,\alpha u +\beta v,\beta u,d)=(u,v;a,b,c,d)$ so $\F(0,0,5)=\{bvu=au^2+cv^2\}\subset \F(0,0,0,6)$. The surface $V_0$ is then given by the equation 
$$\alpha^2(Au^2+Dv^2)+\alpha\beta(2Auv)+\beta^2(Av^2+Bu^2)+d^2g_{12}(u,v) =0,$$
and in $\F(0,0,5)$ we blew up two points, $(u_i,v_i;1,0,0)$ where $Au_i^2+Dv_i^2=0$, to obtain a double cover of $\F_5$. We now refer to these two points by $r_1$ and $r_2$. The surface $V_1=\PS^1\times \PS^1$ is defined by the three equations 
$$0=bz-ay-cw=by-ax-cz = Ab^2+Bc^2+Da^2 \subset \mathbb{P}^3\times \mathbb{P}^2.$$
On $V_0$, the elliptic double curve $E:=V_0\cap V_1$ is given by the equation $d=0$ and $E$ is given by $(x,y,z,w)=(u^3,u^2v,uv^2,v^3)$ on $V_1$. In particular, we see that the points $r_i\in E$ are given by $\left((u_i^3,u_i^2v_i,u_iv_i^2,v_i^3),(v_i,u_i,0)\right)$ for $i=1,2$. We will show that the relation $3q=p_1+r_1+r_2$ is satisfied in $V_0$.

As mentioned in Section \ref{sect:non-normal along TC}, the two rulings arise from choosing suitable $(a,b,c)\in\PS^2$ and from choosing $(z,w)\in\PS^1$. We denote these rulings by $|s|$ and $|f|$ respectively. We consider the two fibers $f_i$ obtained by setting $(z,w)=(u_i,v_i)\in\PS^1$ so $f_1\cap E=\{r_1,r_3\}$ and $f_2\cap E=\{r_2,r_4\}$. We also consider the two sections $s_i$ which arise from the choice $(a,b,c)=(v_i,u_i,0)\in\PS^2$ so $s_1\cap E=\{r_1,r_5\}$ and $s_2\cap E=\{r_2,\ r_6\}$ where $r_5=((0,0,0,1),(v_1,u_1,0))$ and $r_6=((0,0,0,1),(v_2,u_2,0))$. Notice that $r_5$ and $r_6$ belong to the same fiber and therefore satisfy $r_5+r_6+p_s=3q'$ in the group law of the elliptic curve.

We also consider the locus $\{Ax+Dz=0\}$ on $V_1$, which belongs to the linear system $|2s+f|$ and intersects the elliptic curve in all of the $r_i$ defined above. Considering $f_1+f_2+s_1+s_2$ and the locus $\{Ax+Dz=0\}$ shows that $r_1+r_2=3q'-p_s=r_5+r_6$ in the group law of the elliptic curve. 

Considering the points $r_5$ and $r_6$ in $V_0$, we see that they also belong to a fiber in $V_0$. After performing the necessary contractions to $\PS^2$ as in Proposition \ref{double cover Hirzeburch is blowup of P2} we see that $r_5+r_6=3q-p_1$ in the group law of the elliptic curve. Therefore we also see that $r_1+r_2=3q-p_1$. We obtain the desired relation by substituting this into the relation of Proposition \ref{double cover Hirzeburch is blowup of P2}.
\end{proof}

\begin{lemma}
In the $D_{17}$ case (Section \ref{sect:hyperelliptic degen}) the period map imposes the condition $45q=11p_1+2\sum_{i=2}^{18}p_i$ on the points we blowup to obtain $\operatorname{Bl}_{18}\PS^2$. This relation is also determined by the geometry of the degeneration.
\end{lemma}

\begin{proof}
Recall in this case that $h = (3(l-e_1),2l')$ and $\xi=(-3l+\sum_{i=1}^{18}e_i,3l')$. Applying the period map to $h$ and $\xi$ one sees that
$$9q-3p_1-6q'=0, \qquad 9q+9q'-\sum_{i=1}^{18}p_i=0$$
from which one can derive the desired relation.

From a geometric point of view, recall that in Section \ref{sect:hyperelliptic degen} we defined $V_0$ to be a subvariety of $\F(-6,0,-5)$ with coordinates $(u,v;a,b,\alpha)$. We denote the three points $(u_i,v_i;0,0,1)$ where $(u_i,v_i)$ are solutions of the equation $L_3:=Auv^2+u^2v-\frac{k}{2}v^3=0$ by $r_1,r_2$ and $r_3$. After blowing up $r_1,r_2$ and $r_3$, the surface  $\operatorname{Bl}_3V_0$ can be viewed as a double cover of $\F_6$ by Lemma \ref{lemma: blow up hyperelliptic case}. The coordinates on $\F_6$ are $(u,v;a,b)$ and the branch curve is given by the locus $$\{b^2u^2(u+v)^2f_6^2-4L_3(a^2u+b^2(ug_{12}+v^4Q_6L_3))=0\}.$$

We first show that $3q=p_1+2r_i$ for $i=1,2,3$. The unique negative section on $\F_6$ is given by $\{b=0\}$ and intersects the branch locus along $b=uL_3=0$. Furthermore, the three fibers in $\F_6$ arising from the solutions of $L_3$ are tangent to the branch curve at $b=0$. The pullback of the negative section is anti-canonical and the pullback of each of these three fibers is reducible. After contracting as in Proposition \ref{double cover Hirzeburch is blowup of P2} we first obtain 3 fibers in $\F_1$ which are tangent to the anti-canonical curve. When contracting to $\PS^2$ we obtain three lines, each of which is tangent to a cubic curve in a point corresponding to $r_i$ and also passes through $p_1$. This proves the claim. Using this in conjunction with the relation given by Proposition \ref{double cover Hirzeburch is blowup of P2} gives the desired relation.
 \end{proof}
 
The final three cases are those which contained a lattice of type $E_7$ or $E_8$, due to the construction of the hyperplane sections in these cases, the geometric relations are almost immediate.

\begin{lemma}
In the $E_8+D_9$ case (Section \ref{sect: E8 singularity and no line}) the period map imposes the condition $21q = 3p_1+2\sum_{i=2}^{10}p_i$ on the points we blowup to obtain $\operatorname{Bl}_{10}\PS^2$. This relation is also determined by the geometry of the degeneration.
\end{lemma}

\begin{proof}
Reacll that $h=(7l-3e_1-2\sum_{i=2}^{10}e_i,0)$ and $\xi=(-3l+\sum_{i=1}^{10}e_i,3l'-\sum_{i=1}^{8}e_i')$. The period map applied to $h$ and $\xi$ gives the relations
$$21q-3p_1-2\sum_{i=2}^{10}p_i=0 ,\qquad 9q-\sum_{i=1}^{10}p_i+9q'-\sum_{i=1}^{8}p_i'=0$$
respectively. Determining the relation geometrically is almost tautological as we know from the construction of the hyperplane section in Lemma \ref{lemma:root lattice E8D9} that there is a degree 7 curve in $\PS^2$ with the desired properties.
\end{proof}

\begin{lemma}
In the $E_7+E_7+A_3$ case (Section \ref{sect: E7 singularity}) the period map imposes the condition $18q=2\sum_{i=1}^{7}p_i+\sum_{i=8}^{11}p_i$ on the points we blowup to obtain $\operatorname{Bl}_{4}\operatorname{dP}_2$. This relation is also determined by the geometry of the degeneration.
\end{lemma}

\begin{proof}
Recall in this case that $h= (6l-2\sum_{i=1}^7e_i-\sum_{i=8}^{11}e_i,0)$ and $\xi=(-3l+\sum_{i=1}^{11}e_i,3l'-\sum_{i=1}^{7}e_i')$. Applying the period map to $h$ and $\xi$ gives rise to the relations
$$18q-2\sum_{i=1}^7p_i-\sum_{i=8}^{11}p_i=0 , \qquad 9q-\sum_{i=1}^{11}p_i+9q'-\sum_{i=1}^{7}p_i'=0$$
respectively. Again, the relation is determined geometrically in a tautological manner. Constructing the hyperplane section in Lemma \ref{lemma: root lattice E7E7A3} reveals the existence of a degree 6 curve in $\PS^2$ with the desired properties.
\end{proof}

\begin{lemma}
In the $E_8+E_8+\langle -4 \rangle$ case (Section \ref{sect: E8 singularity and a line}) the period map imposes the condition $27q=3\sum_{i=1}^{8}p_i+2p_9+p_{10}$ on the points we blowup to obtain $\operatorname{Bl}_{10}\PS^2$. This relation is also determined by the geometry of the degeneration.
\end{lemma}

\begin{proof}
Recall that $h=(9l-3\sum_{i=1}^8e_i-2e_9-e_{10},0)$ and $\xi=(-3l+\sum_{i=1}^{10}e_i,3l'\sum_{i=1}^{8}e_i')$. The period map applied to $h$ and $\xi$ give rise to the relations
$$27q-3\sum_{i=1}^{8}p_i-2p_9-p_{10}=0 ,\qquad 9q -\sum_{i=1}^{10}p_i+9q'-\sum_{i=1}^{8}p_i'=0$$
respectively. Also in this case, given the class of a hyperplane section as found in Lemma \ref{lemma:root lattice E8E8} we see that there is a degree $9$ curve in $\PS^2$ with the required properties.
\end{proof}

This concludes the proof of Proposition \ref{prop:relations}.

\begin{rk}
The three cases $D_8+D_8+\langle-4\rangle$, $E_7+E_7+A_3$, $E_8+E_8+\langle-4\rangle$  can also be viewed as three component degenerations. Generally they are more difficult to work with in their three component form. Obtaining the relations as we have done in this section is especially non-trivial in the $E_8+E_8+\langle-4\rangle$ case.
\end{rk}

This explicitly shows for each Type II boundary component of $\overline{\mathcal{F}_4}^{tor}$ there is a $\langle 4 \rangle$-quasipolarized $\lambda$-family lying over it. The proof of this is an easy adaptation of the proof of \cite[Theorem 7.23]{AE23}, using Construction \ref{Tyurin constructions} and the relevant relation of Table \ref{Complete Description Table}. The important consequence of this is that the period map is surjective on each of the Type II boundary components of $\overline{\mathcal{F}_4}^{tor}$ because we have constructed families in each case where it is the identity. From this we continue by providing a complete list for all stable models of the central fibers of Tyurin degenerations of degree 4 K3 surfaces.

\section{The Classification of Models}\label{sect:clasification of models}
By a remark of Friedman \cite[Remark 3.9]{Fried84}, any Type II degeneration of K3 surfaces of degree 4 can be obtained from a Tyurin degeneration of degree 4 by performing base changes, flips, and twists. In this section we prove Theorem \ref{list of nef models} which describes the central fibers of all stable models of Tyurin degenerations of degree 4, distinguishing them from a geometric point of view. In addition we can see that up to relabelling, there are only two Tyurin degenerations of degree 4 which admit multiple stable models. These arise from the boundary components $E_6+A_{11}$ and $E_8+E_8+\langle -4 \rangle$. Table \ref{Complete Description Table} also demonstrates several key differences between the $D_8+D_8+\langle -4 \rangle$ and $E_8+E_8+\langle -4 \rangle$ cases, even though they both admit models which have $\operatorname{Bl}_9\PS^2\cup\operatorname{Bl}_9\PS^2$ as their central fiber. Namely, the relations on the points one blows up is different, but we will see another difference in the proof of Theorem \ref{list of nef models} and in Figure \ref{fig:figures}. The remainder of this section is dedicated to proving Theorem \ref{list of nef models}. We first prove a technical lemma which is used in the proof of Theorem \ref{list of nef models}.


\begin{lemma}\label{sections on central fiber come from global}
Let $\pi:X \rightarrow \Delta$ be a Tyurin degeneration K3 surfaces with $L$ a nef line bundle on $X$ such that $L_t$ is nef and big for all fibers. Then there is a surjection $H^0(X,L) \twoheadrightarrow H^0(X_0,L_0)$.
\end{lemma}

\begin{proof}
Note that by a result of Shepherd-Barron \cite[Corollary 3.1]{SB83}, $L$ can always be arranged to be nef. In any case we can use a relative version of Kawamata-Viehweg vanishing \cite[Theorem 2.1]{Ancona87}, as $L$ is relatively nef and is big on the generic fiber. Therefore $R^q\psi_*(L) = 0$ for $q \geq 1$ because $K_{X}=0$. Similarly, we see that $H^q(X_t,L_t)=0$ for $q\geq 1$.

We now use a complex analytic version for cohomology and base change. We use this as opposed to the algebraic setting because in general a Kulikov model may be analytic and non-algebraic. A reference for the results in the analytic setting can be found in \cite[Chapter III.3]{GlobalTheory}, Corollary 3.7 is especially relevant and all the pre-requisites are satisfied. As $L$ is a line bundle, it is locally free and hence flat over $X$ by \cite[III.9.2]{Har77}. Therefore $L$ is flat over $\Delta$ by the transitivity of flatness.

From this we consider the map
$$0=R^1\psi_*(L)_t\rightarrow R^1\psi_*(L/\hat{\mathfrak{m}}_tL)_t=H^1(X_t,L_t)=0$$
which is surjective and note that  $R^1\psi_*(L)_t=0$ is locally free. Therefore by \cite[Corollary III.3.7]{GlobalTheory} we see that the map
$$R^0\psi_*(L)_t\rightarrow R^0\psi_*(L/\hat{\mathfrak{m}}_tL)_t=H^0(X_t,L_t)$$
is surjective.

The stalk $R^0\psi_*(L)_t \simeq H^0(X_t,L)$ using \cite[Lemma III.1.3]{GlobalTheory}. Then as $L$ is nef it is globally generated, and so the natural map $H^0(X,L)\rightarrow H^0(X_t,L)$ is surjective. Therefore we have a composition of surjective maps $$H^0(X,L)\rightarrow H^0(X_t,L)=R^0\psi_*(L)_t\rightarrow R^0\psi_*(L/\hat{\mathfrak{m}}_tL)_t=H^0(X_t,L_t)$$ which gives the desired conclusion when applied to the fiber at $t=0$.
\end{proof}
As a consequence of Lemma \ref{sections on central fiber come from global}, given a nef and relatively big line bundle $L$ on a Tyurin degeneration, an effective divisor in the linear system $|L_0|$ can be extended to an effective divisor in $|L|$. Furthermore, in the case that $L_t$ is a polarization for $t \neq 0$, effective members of $|L_0|$ are guaranteed by upper semicontinuity.
We now prove a result which determines how two nef models are related.

\begin{lemma}\label{nef polarisations are related}
Let $\pi:X\rightarrow \Delta$ be a Tyurin degeneration of K3 surfaces. Let $L,L'\in \operatorname{Pic}(X)$ be two line bundles and suppose there exist integers $m,m'\in\Z$ such that $L_t^m=(L_t')^{m'}$ for all $t\neq0$. Then there exists an integer $n\in\Z$ such that $(L')^{m'}\simeq L^m\otimes\mathcal{O}_X(nV_0)$.
\end{lemma}

\begin{proof}
As $L_t^m=(L_t')^{m'}$ for all $t\neq0$, we see that $L^m$ and $(L')^{m'}$ differ by a line bundle which is supported on the central fiber. Since $\mathcal{O}_X(X_0)$ is trivial, they must differ by $\mathcal{O}_X(nV_0)$ for some integer $n$. Therefore $(L')^m\simeq L^m\otimes\mathcal{O}_X(nV_0)$ as required.
\end{proof}

Notice that if we define $h := L_0$ and $\xi=\mathcal{O}_X(V_0)|_{X_0}$ we see that $(L'_0)^{m'}=mh+n\xi$. As a consequence of this we form the following definition.

\begin{defn}
Given $X\rightarrow \Delta$ a Tyurin degeneration with central fiber $X_0 = V_0\cup V_1$. Let $L$ be a polarization, with $h:=L_0$ and $\xi = \mathcal{O}_X(V_0)|_{X_0}$. We define \emph{lifted polarizations} to be classes on $X_0$ of the form $mh+n\xi$ for integers $m,n$. We define the set of lifted polarizations to be $$\operatorname{Lift}(X_0):=\{mh+n\xi \text{ }\big{|} \text{ }m,n\in\Z\}\subset H^2(X_0,\Z).$$ An \emph{effective lifted polarization} is a lifted polarization which contains an effective divisor. We denote the set of effective lifted polarizations by $\operatorname{Lift_{\geq0}}(X_0)$.
\end{defn}

From now on we refer to the central fibers of the Tyurin degenerations constructed in Section \ref{sect:Explicit descriptions of families} as \emph{standard models} and $X_0$ is used to denote the central fiber of a standard model. To prove Theorem \ref{list of nef models}, we utilise the standard models and the additional relations we constructed in Section \ref{sect:period domains}. For each standard model we construct a decomposition of the cone $\operatorname{Lift}_{\geq0}(X_0)$ in the following way. 

First, let the vertex $(m,n)$ be used to denote the lifted polarization $mh+n\xi$ and let $H_{m,n}\in H^0(X_0,mh+n\xi)$ be an effective divisor. Either $H_{m,n}$ is nef or it is not. If $H_{m,n}$ is nef then we construct the stable model associated to $(X_0,H_{m,n})$, which we denote by $\left(\overline{X_{m,n}},\overline{H_{m,n}}\right)$. Note that if $H_{m,n}$ is ample then $\left(\overline{X_{m,n}},\overline{H_{m,n}}\right)=(X_0,H_{m,n})$. If $H_{m,n}$ is not nef then by a result of Shepherd-Barron \cite[Theorem 1]{SB83} there is a sequence of flops $\phi:X_0\dashrightarrow X_{m,n}$ such that the strict transform $\phi_*H_{m,n}$ is nef. We then construct the stable model associated to $(X_{m,n},\phi_*H_{m,n})$ which we also denote by $\left(\overline{X_{m,n}},\overline{H_{m,n}}\right)$. Note once again that if $\phi_*H_{m,n}$ is ample, then $\left(\overline{X_{m,n}},\overline{H_{m,n}}\right)=(X_{m,n},\phi_*H_{m,n})$. 

The point $(m,n)$ belongs to the interior of a chamber if either $H_{m,n}$ is ample on $X_0$ or there is a sequence of flops $\phi:X_0\dashrightarrow X_{m,n}$ such that $\phi_*H_{m,n}$ is ample. The label of this chamber is $X_0$ if $H_{m,n}$ is ample or $X_{m,n}$ if instead $\phi_*H_{m,n}$ is ample. As the figures are 2-dimensional and the ample cone is convex, the chambers meet along rays which correspond to the situation that either $H_{m,n}$ or $\phi_*H_{m,n}$ is nef because the nef cone is the closure of the ample cone. Notice that the stable model associated to an internal ray is obtained by contracting curves on the stable model of a neighbouring chamber. Moreover, two chambers which meet along a ray are related by a flop. Finally, the stable model associated to a point on an exterior ray contracts a component of the stable model of the neighbouring chamber to a point.

In the diagrams of Figure \ref{fig:figures} the rays correspond to the positive scalar multiples of a given vertex. For example in Figure \ref{nef cone A1A1A15} the ray labeled $2h+\xi$ denotes positive multiples of the point $(2,1)$. In each diagram the point $(1,0)$ gives rise to the singular degenerations we used in Section \ref{sect:Explicit descriptions of families}. The wall and chamber decompositions of $\operatorname{Lift_{\geq0}}(X_0)$ for each case are provided in Figure \ref{fig:figures} and the proofs for these decompositions are provided in Lemmas \ref{lifted polarization A1A1A15}--\ref{lifted polarization E8E8-4}.

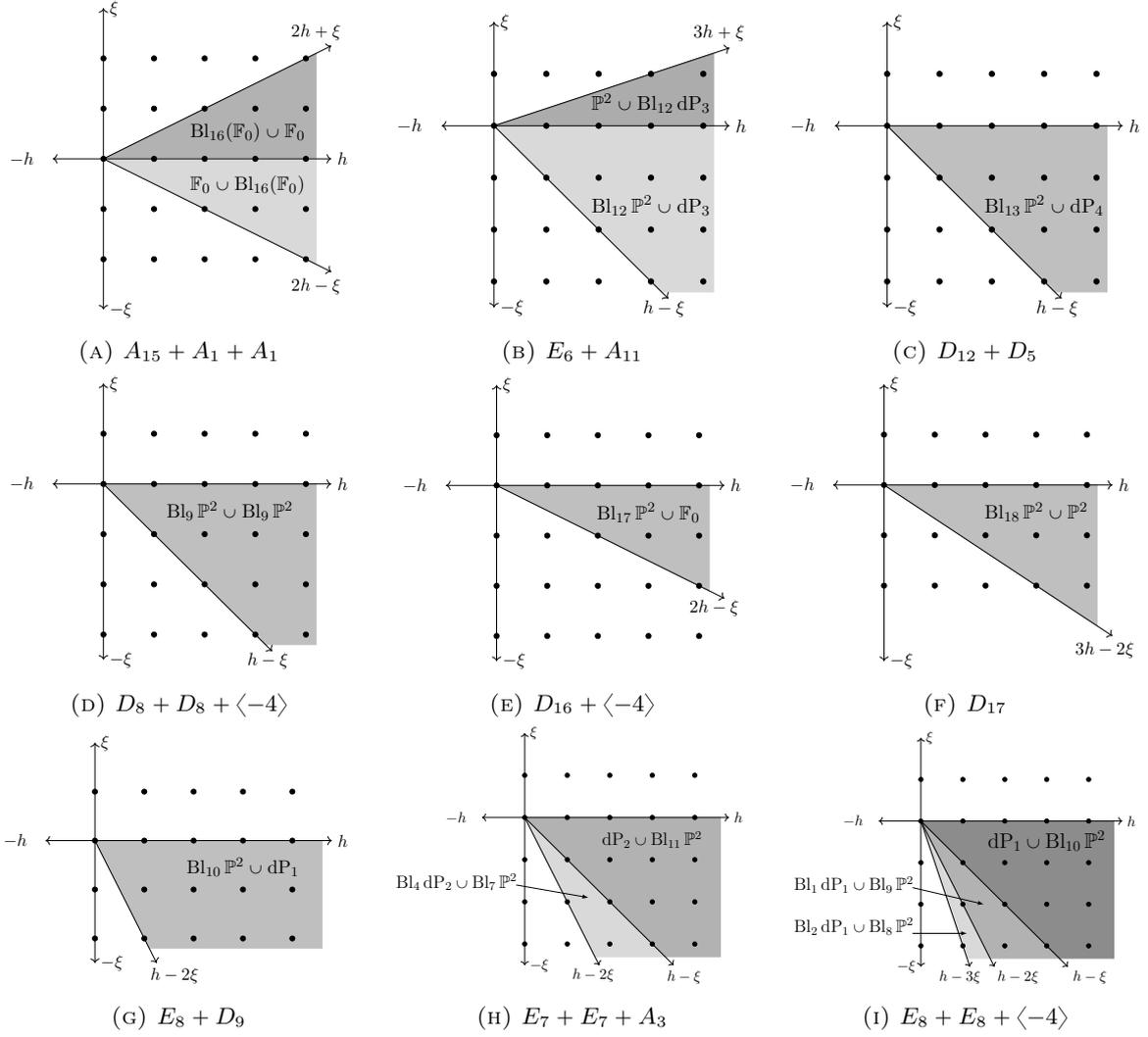
\begin{figure}[h]
\centering

\begin{subfigure}{0.315\textwidth}
\resizebox{\textwidth}{!}{
\begin{tikzpicture}
\draw [gray!60, fill=gray!60] (0,0)--(4.2,0)--(4.2,2.1)--(0,0);
\draw [gray!30, fill=gray!30] (0,0)--(4.2,0)--(4.2,-2.1)--(0,0);
\draw [->] (0,0)--(4.5,0);
\draw [->] (0,0)--(-1,0);
\draw [->] (0,0)--(0,3);
\draw [->] (0,0)--(0,-3);
\draw [->] (0,0)--(4.5,2.25);
\draw [->] (0,0)--(4.5,-2.25);
\draw [fill] (0,0) circle [radius=0.05]; 
\draw [fill] (0,1) circle [radius=0.05]; 
\draw [fill] (0,2) circle [radius=0.05]; 
\draw [fill] (1,0) circle [radius=0.05]; 
\draw [fill] (1,1) circle [radius=0.05]; 
\draw [fill] (1,2) circle [radius=0.05]; 
\draw [fill] (2,0) circle [radius=0.05]; 
\draw [fill] (2,1) circle [radius=0.05]; 
\draw [fill] (2,2) circle [radius=0.05];  
\draw [fill] (3,0) circle [radius=0.05]; 
\draw [fill] (3,1) circle [radius=0.05]; 
\draw [fill] (3,2) circle [radius=0.05];  
\draw [fill] (4,0) circle [radius=0.05];
\draw [fill] (4,1) circle [radius=0.05];
\draw [fill] (4,2) circle [radius=0.05];
\draw [fill] (0,-1) circle [radius=0.05]; 
\draw [fill] (0,-2) circle [radius=0.05]; 
\draw [fill] (1,-1) circle [radius=0.05];
\draw [fill] (1,-2) circle [radius=0.05]; 
\draw [fill] (2,-1) circle [radius=0.05]; 
\draw [fill] (2,-2) circle [radius=0.05];  
\draw [fill] (3,-1) circle [radius=0.05]; 
\draw [fill] (3,-2) circle [radius=0.05]; 
\draw [fill] (4,-1) circle [radius=0.05];
\draw [fill] (4,-2) circle [radius=0.05];
\node [right] at (0,3) {$\xi$};
\node [right] at (0,-3) {$-\xi$};
\node [right] at (4.5,0) {$h$};
\node [left] at (-1.25,0) {$-h$};
\node [above] at (4.2,2.25) {$2h+\xi$};
\node [below] at (4.2,-2.25) {$2h-\xi$};
\node [below] at (2.85,0.85) {\large $\operatorname{Bl}_{16}(\F_0)\cup \F_0$};
\node [above] at (2.85,-0.85) {\large $\F_0\cup\operatorname{Bl}_{16}(\F_0)$};
\end{tikzpicture}
}
\subcaption{$A_{15}+A_1+A_1$}
\label{nef cone A1A1A15}
\end{subfigure}
\hfill
\begin{subfigure}{0.315\textwidth}
\resizebox{\textwidth}{!}{\begin{tikzpicture}
\draw [gray!30, fill=gray!30] (0,0)--(3.2,-3.2)--(4.2,-3.2)--(4.2,0)--(0,0);
\draw [gray!65, fill=gray!65] (0,0)--(4.2,0)--(4.2,1.4)--(0,0);
\draw [->] (0,0)--(4.5,0);
\draw [->] (0,0)--(-1,0);
\draw [->] (0,0)--(0,-3.5);
\draw [->] (0,0)--(0,2);
\draw [->] (0,0)--(3.33,-3.33);
\draw [->] (0,0)--(4.5,1.5);
\draw [fill] (0,0) circle [radius=0.05]; 
\draw [fill] (0,1) circle [radius=0.05]; 
\draw [fill] (1,0) circle [radius=0.05]; 
\draw [fill] (1,1) circle [radius=0.05];  
\draw [fill] (2,0) circle [radius=0.05]; 
\draw [fill] (2,1) circle [radius=0.05];   
\draw [fill] (3,0) circle [radius=0.05]; 
\draw [fill] (3,1) circle [radius=0.05];   
\draw [fill] (4,0) circle [radius=0.05];
\draw [fill] (4,1) circle [radius=0.05];
\draw [fill] (0,-1) circle [radius=0.05]; 
\draw [fill] (0,-2) circle [radius=0.05]; 
\draw [fill] (0,-3) circle [radius=0.05]; 
\draw [fill] (1,-1) circle [radius=0.05];
\draw [fill] (1,-2) circle [radius=0.05];
\draw [fill] (1,-3) circle [radius=0.05]; 
\draw [fill] (2,-1) circle [radius=0.05]; 
\draw [fill] (2,-2) circle [radius=0.05];  
\draw [fill] (2,-3) circle [radius=0.05]; 
\draw [fill] (3,-1) circle [radius=0.05]; 
\draw [fill] (3,-2) circle [radius=0.05]; 
\draw [fill] (3,-3) circle [radius=0.05]; 
\draw [fill] (4,-1) circle [radius=0.05];
\draw [fill] (4,-2) circle [radius=0.05];
\draw [fill] (4,-3) circle [radius=0.05]; 
\node [right] at (0,2) {$\xi$};
\node [right] at (0,-3.5) {$-\xi$};
\node [right] at (4.5,0) {$h$};
\node [left] at (-1.25,0) {$-h$};
\node [below] at (3.25,-3.25) {$h-\xi$};
\node [above] at (4.3,1.5) {$3h+\xi$};
\node [above] at (3,-1.85) {\large$\operatorname{Bl}_{12}\PS^2\cup\operatorname{dP}_{3}$};
\node [below] at (3,0.75) {\large$\PS^2\cup\operatorname{Bl}_{12}\operatorname{dP}_3$};
\end{tikzpicture}}
\subcaption{$E_6+A_{11}$}
\label{nef cone A11E6}
\end{subfigure}
\hfill
\begin{subfigure}{0.315\textwidth}
\resizebox{\textwidth}{!}{\begin{tikzpicture}
\draw [lightgray, fill=lightgray] (0,0)--(3.2,-3.2)--(4.2,-3.2)--(4.2,0)--(0,0);
\draw [->] (0,0)--(4.5,0);
\draw [->] (0,0)--(-1,0);
\draw [->] (0,0)--(0,-3.5);
\draw [->] (0,0)--(0,2);
\draw [->] (0,0)--(3.33,-3.33);
\draw [fill] (0,0) circle [radius=0.05]; 
\draw [fill] (0,1) circle [radius=0.05]; 
\draw [fill] (1,0) circle [radius=0.05]; 
\draw [fill] (1,1) circle [radius=0.05];  
\draw [fill] (2,0) circle [radius=0.05]; 
\draw [fill] (2,1) circle [radius=0.05];   
\draw [fill] (3,0) circle [radius=0.05]; 
\draw [fill] (3,1) circle [radius=0.05];   
\draw [fill] (4,0) circle [radius=0.05];
\draw [fill] (4,1) circle [radius=0.05];
\draw [fill] (0,-1) circle [radius=0.05]; 
\draw [fill] (0,-2) circle [radius=0.05]; 
\draw [fill] (0,-3) circle [radius=0.05]; 
\draw [fill] (1,-1) circle [radius=0.05];
\draw [fill] (1,-2) circle [radius=0.05];
\draw [fill] (1,-3) circle [radius=0.05]; 
\draw [fill] (2,-1) circle [radius=0.05]; 
\draw [fill] (2,-2) circle [radius=0.05];  
\draw [fill] (2,-3) circle [radius=0.05]; 
\draw [fill] (3,-1) circle [radius=0.05]; 
\draw [fill] (3,-2) circle [radius=0.05]; 
\draw [fill] (3,-3) circle [radius=0.05]; 
\draw [fill] (4,-1) circle [radius=0.05];
\draw [fill] (4,-2) circle [radius=0.05];
\draw [fill] (4,-3) circle [radius=0.05]; 
\node [right] at (0,2) {$\xi$};
\node [right] at (0,-3.5) {$-\xi$};
\node [right] at (4.5,0) {$h$};
\node [left] at (-1.25,0) {$-h$};
\node [below] at (3.25,-3.25) {$h-\xi$};
\node [above] at (3,-1.85) {\large$\operatorname{Bl}_{13}\PS^2\cup\operatorname{dP}_{4}$};
\end{tikzpicture}}
\subcaption{$D_{12}+D_5$}
\label{nef cone D12D5}
\end{subfigure}
\hfill
\begin{subfigure}{0.315\textwidth}
\resizebox{\textwidth}{!}{
\begin{tikzpicture}
\draw [lightgray, fill=lightgray] (0,0)--(3.2,-3.2)--(4.2,-3.2)--(4.2,0)--(0,0);
\draw [->] (0,0)--(4.5,0);
\draw [->] (0,0)--(-1,0);
\draw [->] (0,0)--(0,-3.5);
\draw [->] (0,0)--(0,2);
\draw [->] (0,0)--(3.33,-3.33);
\draw [fill] (0,0) circle [radius=0.05]; 
\draw [fill] (0,1) circle [radius=0.05]; 
\draw [fill] (1,0) circle [radius=0.05]; 
\draw [fill] (1,1) circle [radius=0.05];  
\draw [fill] (2,0) circle [radius=0.05]; 
\draw [fill] (2,1) circle [radius=0.05];   
\draw [fill] (3,0) circle [radius=0.05]; 
\draw [fill] (3,1) circle [radius=0.05];   
\draw [fill] (4,0) circle [radius=0.05];
\draw [fill] (4,1) circle [radius=0.05];
\draw [fill] (0,-1) circle [radius=0.05]; 
\draw [fill] (0,-2) circle [radius=0.05]; 
\draw [fill] (0,-3) circle [radius=0.05]; 
\draw [fill] (1,-1) circle [radius=0.05];
\draw [fill] (1,-2) circle [radius=0.05];
\draw [fill] (1,-3) circle [radius=0.05]; 
\draw [fill] (2,-1) circle [radius=0.05]; 
\draw [fill] (2,-2) circle [radius=0.05];  
\draw [fill] (2,-3) circle [radius=0.05]; 
\draw [fill] (3,-1) circle [radius=0.05]; 
\draw [fill] (3,-2) circle [radius=0.05]; 
\draw [fill] (3,-3) circle [radius=0.05]; 
\draw [fill] (4,-1) circle [radius=0.05];
\draw [fill] (4,-2) circle [radius=0.05];
\draw [fill] (4,-3) circle [radius=0.05]; 
\node [right] at (0,2) {$\xi$};
\node [right] at (0,-3.5) {$-\xi$};
\node [right] at (4.5,0) {$h$};
\node [left] at (-1.25,0) {$-h$};
\node [below] at (3.25,-3.25) {$h-\xi$};
\node [above] at (2.5,-0.85) {\large$\operatorname{Bl}_{9}\PS^2\cup\operatorname{Bl}_{9}\PS^2$};
\end{tikzpicture}
}
\subcaption{$D_8+D_8+\langle-4\rangle$}
\label{nef cone D8D8}
\end{subfigure}
\hfill
\begin{subfigure}{0.315\textwidth}
\resizebox{\textwidth}{!}{
\begin{tikzpicture}
\draw [lightgray, fill=lightgray] (0,0)--(4.2,-2.1)--(4.2,0)--(0,0);
\draw [->] (0,0)--(4.5,0);
\draw [->] (0,0)--(-1,0);
\draw [->] (0,0)--(0,-3.5);
\draw [->] (0,0)--(0,2);
\draw [->] (0,0)--(4.5,-2.25);
\draw [fill] (0,0) circle [radius=0.05]; 
\draw [fill] (0,1) circle [radius=0.05]; 
\draw [fill] (1,0) circle [radius=0.05]; 
\draw [fill] (1,1) circle [radius=0.05];  
\draw [fill] (2,0) circle [radius=0.05]; 
\draw [fill] (2,1) circle [radius=0.05];   
\draw [fill] (3,0) circle [radius=0.05]; 
\draw [fill] (3,1) circle [radius=0.05];   
\draw [fill] (4,0) circle [radius=0.05];
\draw [fill] (4,1) circle [radius=0.05];
\draw [fill] (0,-1) circle [radius=0.05]; 
\draw [fill] (0,-2) circle [radius=0.05];
\draw [fill] (0,-3) circle [radius=0.05]; 
\draw [fill] (1,-1) circle [radius=0.05];
\draw [fill] (1,-2) circle [radius=0.05];
\draw [fill] (1,-3) circle [radius=0.05]; 
\draw [fill] (2,-1) circle [radius=0.05]; 
\draw [fill] (2,-2) circle [radius=0.05];
\draw [fill] (2,-3) circle [radius=0.05]; 
\draw [fill] (3,-1) circle [radius=0.05]; 
\draw [fill] (3,-2) circle [radius=0.05];
\draw [fill] (3,-3) circle [radius=0.05]; 
\draw [fill] (4,-1) circle [radius=0.05];
\draw [fill] (4,-2) circle [radius=0.05];
\draw [fill] (4,-3) circle [radius=0.05]; 
\node [right] at (0,2) {$\xi$};
\node [right] at (0,-3.5) {$-\xi$};
\node [right] at (4.5,0) {$h$};
\node [left] at (-1.25,0) {$-h$};
\node [below] at (4.3,-2.15) {$2h-\xi$};
\node [above] at (3,-0.85) {\large$\operatorname{Bl}_{17}\PS^2\cup\F_0$};
\end{tikzpicture}
}
\subcaption{$D_{16}+\langle -4 \rangle$}
\label{nef cone D16}
\end{subfigure}
\hfill
\begin{subfigure}{0.315\textwidth}
\resizebox{\textwidth}{!}{\begin{tikzpicture}
\draw [lightgray, fill=lightgray] (0,0)--(4.2,-2.8)--(4.2,0)--(0,0);
\draw [->] (0,0)--(4.5,0);
\draw [->] (0,0)--(-1,0);
\draw [->] (0,0)--(0,-3.5);
\draw [->] (0,0)--(0,2);
\draw [->] (0,0)--(4.5,-3);
\draw [fill] (0,0) circle [radius=0.05]; 
\draw [fill] (0,1) circle [radius=0.05]; 
\draw [fill] (1,0) circle [radius=0.05]; 
\draw [fill] (1,1) circle [radius=0.05];  
\draw [fill] (2,0) circle [radius=0.05]; 
\draw [fill] (2,1) circle [radius=0.05];   
\draw [fill] (3,0) circle [radius=0.05]; 
\draw [fill] (3,1) circle [radius=0.05];   
\draw [fill] (4,0) circle [radius=0.05];
\draw [fill] (4,1) circle [radius=0.05];
\draw [fill] (0,-1) circle [radius=0.05]; 
\draw [fill] (0,-2) circle [radius=0.05]; 
\draw [fill] (1,-1) circle [radius=0.05];
\draw [fill] (1,-2) circle [radius=0.05];
\draw [fill] (2,-1) circle [radius=0.05]; 
\draw [fill] (2,-2) circle [radius=0.05]; 
\draw [fill] (3,-1) circle [radius=0.05]; 
\draw [fill] (3,-2) circle [radius=0.05];
\draw [fill] (4,-1) circle [radius=0.05];
\draw [fill] (4,-2) circle [radius=0.05]; 
\node [right] at (0,2) {$\xi$};
\node [right] at (0,-3.5) {$-\xi$};
\node [right] at (4.5,0) {$h$};
\node [left] at (-1.25,0) {$-h$};
\node [below] at (4.35,-3) {$3h-2\xi$};
\node [above] at (3,-0.85) {\large$\operatorname{Bl}_{18}\PS^2\cup\PS^2$};
\end{tikzpicture}
}
\subcaption{$D_{17}$}
\label{nef cone D17}
\end{subfigure}
\hfill
\begin{subfigure}{0.315\textwidth}
\resizebox{\textwidth}{!}{\begin{tikzpicture}
\draw [lightgray, fill=lightgray] (0,0)--(1.1,-2.2)--(4.6,-2.2)--(4.6,0)--(0,0);
\draw [->] (0,0)--(4.8,0);
\draw [->] (0,0)--(-1,0);
\draw [->] (0,0)--(0,-2.5);
\draw [->] (0,0)--(0,2);
\draw [->] (0,0)--(1.25,-2.5);
\draw [fill] (0,0) circle [radius=0.05]; 
\draw [fill] (0,1) circle [radius=0.05]; 
\draw [fill] (1,0) circle [radius=0.05]; 
\draw [fill] (1,1) circle [radius=0.05];  
\draw [fill] (2,0) circle [radius=0.05]; 
\draw [fill] (2,1) circle [radius=0.05];   
\draw [fill] (3,0) circle [radius=0.05]; 
\draw [fill] (3,1) circle [radius=0.05];   
\draw [fill] (4,0) circle [radius=0.05];
\draw [fill] (4,1) circle [radius=0.05];
\draw [fill] (0,-1) circle [radius=0.05]; 
\draw [fill] (0,-2) circle [radius=0.05];  
\draw [fill] (1,-1) circle [radius=0.05];
\draw [fill] (1,-2) circle [radius=0.05]; 
\draw [fill] (2,-1) circle [radius=0.05]; 
\draw [fill] (2,-2) circle [radius=0.05];
\draw [fill] (3,-1) circle [radius=0.05]; 
\draw [fill] (3,-2) circle [radius=0.05]; 
\draw [fill] (4,-1) circle [radius=0.05];
\draw [fill] (4,-2) circle [radius=0.05]; 
\node [right] at (0,2) {$\xi$};
\node [right] at (0,-2.5) {$-\xi$};
\node [right] at (4.8,0) {$h$};
\node [left] at (-1.25,0) {$-h$};
\node [right] at (1,-2.75) {$h-2\xi$};
\node [above] at (3,-0.85) {\large$\operatorname{Bl}_{10}\PS^2\cup\operatorname{dP}_1$};
\end{tikzpicture}
}
\subcaption{$E_8+D_9$}
\label{nef cone E8D9}
\end{subfigure}
\hfill
\begin{subfigure}{0.315\textwidth}
\resizebox{\textwidth}{!}{\begin{tikzpicture}
\draw [gray!30, fill=gray!30] (0,0)--(1.65,-3.3)--(3.3,-3.3)--(0,0);
\draw [gray!60, fill=gray!60] (0,0)--(3.3,-3.3)--(4.6,-3.3)--(4.6,0)--(0,0);
\draw [->] (0,0)--(4.8,0);
\draw [->] (0,0)--(-1,0);
\draw [->] (0,0)--(0,-3.5);
\draw [->] (0,0)--(0,2);
\draw [->] (0,0)--(1.75,-3.5);
\draw [->] (0,0)--(3.5,-3.5);
\draw [-latex] (-0.2,-1.6)--(1.5,-1.9);
\draw [fill] (0,0) circle [radius=0.05]; 
\draw [fill] (0,1) circle [radius=0.05]; 
\draw [fill] (1,0) circle [radius=0.05]; 
\draw [fill] (1,1) circle [radius=0.05];  
\draw [fill] (2,0) circle [radius=0.05]; 
\draw [fill] (2,1) circle [radius=0.05];   
\draw [fill] (3,0) circle [radius=0.05]; 
\draw [fill] (3,1) circle [radius=0.05];   
\draw [fill] (4,0) circle [radius=0.05];
\draw [fill] (4,1) circle [radius=0.05];
\draw [fill] (0,-1) circle [radius=0.05]; 
\draw [fill] (0,-2) circle [radius=0.05];
\draw [fill] (0,-3) circle [radius=0.05];
\draw [fill] (1,-1) circle [radius=0.05];
\draw [fill] (1,-2) circle [radius=0.05];
\draw [fill] (1,-3) circle [radius=0.05];
\draw [fill] (2,-1) circle [radius=0.05]; 
\draw [fill] (2,-2) circle [radius=0.05];
\draw [fill] (2,-3) circle [radius=0.05];
\draw [fill] (3,-1) circle [radius=0.05]; 
\draw [fill] (3,-2) circle [radius=0.05]; 
\draw [fill] (3,-3) circle [radius=0.05];
\draw [fill] (4,-1) circle [radius=0.05];
\draw [fill] (4,-2) circle [radius=0.05]; 
\draw [fill] (4,-3) circle [radius=0.05];
\node [right] at (0,2) {$\xi$};
\node [right] at (0,-3.5) {$-\xi$};
\node [right] at (4.8,0) {$h$};
\node [left] at (-1.25,0) {$-h$};
\node [right] at (1,-3.75) {$h-2\xi$};
\node [below] at (3.75,-3.5) {$h-\xi$};
\node [above] at (3,-0.85) {\Large$\operatorname{dP}_{2}\cup\operatorname{Bl}_{11}\PS^2$};
\node [above] at (-1.6,-1.85) {\Large$\operatorname{Bl}_4\operatorname{dP}_{2}\cup\operatorname{Bl}_{7}\PS^2$};
\end{tikzpicture}
}
\subcaption{$E_7+E_7+A_3$}
\label{nef cone E7E7A3}
\end{subfigure}
\hfill
\begin{subfigure}{0.315\textwidth}
\resizebox{\textwidth}{!}{
\begin{tikzpicture}
\draw [gray!30, fill=gray!30] (0,0)--(1.1,-3.3)--(1.65,-3.3)--(0,0);
\draw [gray!60, fill=gray!60] (0,0)--(1.65,-3.3)--(3.3,-3.3)--(0,0);
\draw [gray!90, fill=gray!90] (0,0)--(3.3,-3.3)--(4.6,-3.3)--(4.6,0)--(0,0);
\draw [->] (0,0)--(4.8,0);
\draw [->] (0,0)--(-1,0);
\draw [->] (0,0)--(0,-3.5);
\draw [->] (0,0)--(0,2);
\draw [->] (0,0)--(1.75,-3.5);
\draw [->] (0,0)--(3.5,-3.5);
\draw [->] (0,0)--(1.167,-3.5);
\draw [-latex] (-0.2,-1.6)--(1.5,-2);
\draw [-latex] (-0.2,-2.7)--(1.1,-2.7);
\draw [fill] (0,0) circle [radius=0.05]; 
\draw [fill] (0,1) circle [radius=0.05]; 
\draw [fill] (1,0) circle [radius=0.05]; 
\draw [fill] (1,1) circle [radius=0.05];  
\draw [fill] (2,0) circle [radius=0.05]; 
\draw [fill] (2,1) circle [radius=0.05];   
\draw [fill] (3,0) circle [radius=0.05]; 
\draw [fill] (3,1) circle [radius=0.05];   
\draw [fill] (4,0) circle [radius=0.05];
\draw [fill] (4,1) circle [radius=0.05];
\draw [fill] (0,-1) circle [radius=0.05]; 
\draw [fill] (0,-2) circle [radius=0.05];
\draw [fill] (0,-3) circle [radius=0.05];
\draw [fill] (1,-1) circle [radius=0.05];
\draw [fill] (1,-2) circle [radius=0.05];
\draw [fill] (1,-3) circle [radius=0.05];
\draw [fill] (2,-1) circle [radius=0.05]; 
\draw [fill] (2,-2) circle [radius=0.05];
\draw [fill] (2,-3) circle [radius=0.05];
\draw [fill] (3,-1) circle [radius=0.05]; 
\draw [fill] (3,-2) circle [radius=0.05]; 
\draw [fill] (3,-3) circle [radius=0.05];
\draw [fill] (4,-1) circle [radius=0.05];
\draw [fill] (4,-2) circle [radius=0.05]; 
\draw [fill] (4,-3) circle [radius=0.05];
\node [right] at (0,2) {$\xi$};
\node [left] at (0,-3.5) {$-\xi$};
\node [right] at (4.8,0) {$h$};
\node [left] at (-1.25,0) {$-h$};
\node [right] at (1.7,-3.75) {$h-2\xi$};
\node [right] at (3.5,-3.75) {$h-\xi$};
\node [right] at (0.315,-3.75){$h-3\xi$};
\node [above] at (3,-0.85) {\LARGE$\operatorname{dP}_{1}\cup\operatorname{Bl}_{10}\PS^2$};
\node [above] at (-1.6,-1.85) {\Large$\operatorname{Bl}_1\operatorname{dP}_{1}\cup\operatorname{Bl}_{9}\PS^2$};
\node [above] at (-1.6,-2.85) {\Large$\operatorname{Bl}_2\operatorname{dP}_{1}\cup\operatorname{Bl}_{8}\PS^2$};
\end{tikzpicture}
}
\subcaption{$E_8+E_8+\langle-4\rangle$}
\label{nef cone E8E8}

\end{subfigure}

\caption{The wall and chamber decompositions for $\operatorname{Lift}_{\geq0}(X_0)$ for each case of degeneration.}
\label{fig:figures}
\end{figure} 

\begin{lemma}\label{lifted polarization A1A1A15}
In the $A_1+A_1+A_{15}$ case $\operatorname{Lift}_{\geq 0}(X_0)$ has two chambers with one interior wall.
\end{lemma}

\begin{proof}
Recall $h=(s_1+f_1,s_2+f_2)$ and $\xi = (-2s_1-2f_1+\sum_{i=1}^{16}e_i,2s_2+2f_2)$, it is quick to check that $\operatorname{Lift}_{\geq 0}(X_0)$ is spanned by $2h+\xi$ and $2h-\xi$. Considering $h'=h+\epsilon\xi$, we see that $h'$ is not nef for $\epsilon > 0 $ and we flip all $e_i$ simultaneously to $V_1$. After performing the flips, the double curve is $\xi_1=(-2s_1-2f_1,2s_2+2f_2-\sum_{i=1}^{16}e_i)$ and we consider $h_1=h+\epsilon\xi_1$. This is not nef for $\epsilon > 1/2$ and the stable model contracts $V_0$ to a point at $\epsilon =1/2$.

Similarly we consider $h'=h-\epsilon\xi$ for $\epsilon > 0$. This is not nef for $\epsilon > 1/2$ where it contracts $V_1$ to a point at $\epsilon = 1/2$. The wall and chamber decomposition is given in Figure \ref{nef cone A1A1A15}.
\end{proof}

Lemma \ref{lifted polarization A1A1A15} is a result of the fact that we could have chosen to resolve all of the singularities in either component when we first constructed this degeneration in Section \ref{sect:Explicit descriptions of families}. The same is true for the next lemma too.
\FloatBarrier
\begin{lemma}
In the $A_{11}+E_6$ case $\operatorname{Lift}_{\geq 0}(X_0)$ has two chambers with one interior wall.
\end{lemma}
\FloatBarrier
\begin{proof}
\FloatBarrier
Recall that $h=(l,3l'-\sum_{i=1}^6e_i')$ and $\xi = (-3l+\sum_{i=1}^{12}e_i,3l'-\sum_{i=1}^6e_i')$, one can check that $\operatorname{Lift}_{\geq 0}(X_0)$ is spanned by $3h+\xi$ and $h-\xi$. Consider $h'=h+\epsilon\xi$ for $\epsilon>0$, then $h'$ is not nef and we must flip all $e_i$ to $V_1$, and define $\xi_1=(-3l,3l'-\sum_{i=1}^6e_i'-\sum_{i=1}^{12}e_i)$. Then $h+\epsilon\xi_1$ is nef until $\epsilon>1/3$ and the stable model contracts $V_0$ to a point at $\epsilon=1/3$.

Similarly, considering $h-\epsilon\xi$ is nef until $\epsilon>1$ and the stable model contracts $V_1$ to a point when $\epsilon=1$. The wall and chamber decomposition is given in Figure \ref{nef cone A11E6}.
\end{proof}
\FloatBarrier
\begin{lemma}
In the $D_{12}+D_5$ case $\operatorname{Lift}_{\geq 0}(X_0)$ has one chamber.
\end{lemma}
\FloatBarrier
\begin{proof}
Recall that $h=(2(l-e_1),3l'-\sum_{i=1}^5e_i')$ and $\xi=(-3l+\sum_{i=1}^{13}e_i,3l-\sum_{i=1}^5e_i')$, it is easy to see that $\operatorname{Lift}_{\geq 0}(X_0)$ is spanned by $h$ and $h-\xi$. Consider $h+\epsilon\xi$ for $\epsilon >0$ then $h'.(l-e_1)<0$ and  $h'.e_1>0$ so we cannot perform any flips to resolve this. Therefore we can only consider $h'=h-\epsilon\xi$ for $\epsilon >0$. Such an $h'$ is nef until $\epsilon >1$ and the stable model contracts $V_1$ to a point when $\epsilon=1$. The wall and chamber decomposition is given in Figure \ref{nef cone D12D5}.
\end{proof}
\FloatBarrier
\begin{lemma}
In the $D_8+D_8+\langle -4\rangle$ case $\operatorname{Lift}_{\geq 0}(X_0)$ has one chamber.
\end{lemma}

\begin{proof}
Recall that $h=(l-e_1,4l'-2e_1'-\sum_{i=2}^9e_i')$ and $\xi=(-3l+\sum_{i=1}^9e_i,3l-\sum_{i=1}^9e_i')$, one can easily see that $\operatorname{Lift}_{\geq 0}(X_0)$ is spanned by $h$ and $h-\xi$. We first consider $h'=h+\epsilon\xi$ for $\epsilon>0$ but notice that $h'. (l-e_1)<0$ and so $h'$ is not effective. We consider instead $h'=h-\epsilon\xi$ for $\epsilon >0$. Such an $h'$ is nef until $\epsilon >1$ where $h'. e_i <0$ for $i\geq 2$, we also see that $h'$ is not effective when $\epsilon >1$ because $h'. (l'-e_1')<0$. Therefore there is just one chamber, see Figure \ref{nef cone D8D8}.
\end{proof}

\begin{lemma}
In the $D_{16}+\langle -4\rangle$ case $\operatorname{Lift}_{\geq 0}(X_0)$ has one chamber.
\end{lemma}

\begin{proof}
Recall that $V_0 = \operatorname{Bl_{17}\PS^2}$ and $V_1=\PS^1\times\PS^1$ with $h=(3(l-e_1),s+2f)$ and $\xi=(-3l+\sum_{i=1}^{17}e_i,2s+2f)$, it is straightforward to check that $\operatorname{Lift}_{\geq 0}(X_0)$ is spanned by $h$ and $2h-\xi$. Considering $h'=h+\epsilon\xi$ for $\epsilon > 0$ we see that $h'.e_1>0$ while $h.(l-e_1)<0$, so $h'$ is not effective and this cannot be resolved by flipping curves. Therefore we can only consider $h'=h-\epsilon\xi$ for $\epsilon >0$. Such an $h'$ is nef until $\epsilon >1/2$, and the stable model contracts the entirety of $V_1$ to a point when $\epsilon =1/2$. The wall and chamber decomposition is given in Figure \ref{nef cone D16}.
\end{proof}

\begin{lemma}
In the $D_{17}$ case $\operatorname{Lift}_{\geq 0}(X_0)$ has one chamber.
\end{lemma}

\begin{proof}
Recall that $h=(3(l-e_1),2l')$ and $\xi=(-3l+\sum_{i=1}^{18}e_i,3l')$ and one can check that $\operatorname{Lift}_{\geq 0}(X_0)$ is spanned by $h$ and $3h-2\xi$. Considering $h'=h+\epsilon\xi$ for $\epsilon >0$ we see that $h'. e_1>0$ while $h'.(l-e_1)<0$, so $h'$ is not effective and this cannot be resolved by flipping curves. Therefore we can only consider $h'$ of the form $h'=h-\epsilon\xi$ where $\epsilon >0$. Such an $h'$ is nef until $\epsilon >2/3$ and the stable model contracts all of $V_1$ to a point when $\epsilon =2/3$. The wall and chamber decomposition is given in Figure \ref{nef cone D17}.
\end{proof}

\begin{lemma}
In the $E_8+D_9$ case $\operatorname{Lift}_{\geq 0}(X_0)$ has one chamber.
\end{lemma}

\begin{proof}
Recall that $h=(0,7l'-3e_1'-2\sum_{i=2}^{10}e_i')$ and $\xi=(-3l+\sum_{i=1}^7e_i,3l'-\sum_{i=1}^{10}e_i')$, one can check that $\operatorname{Lift}_{\geq 0}(X_0)$ is spanned by $h$ and $h-2\xi$. We first consider $h' = h+\epsilon\xi$ for $\epsilon >0$. Such an $h'$ is not effective as $h'. l<0$ which cannot be altered by flipping curves, therefore we can only consider $h'=h-\epsilon\xi$ for $\epsilon >0$.

Set $h'=h-\epsilon\xi$ with $\epsilon >0$, then $h'$ is nef until $\epsilon > 2$, for $\epsilon = 2$ we flip all $e_i'$ for $i\geq 2$. Then $h'.(l'-e_1')<0$, so $h'$ is not effective. Therefore we only have one chamber and the wall and chamber decomposition is given in Figure \ref{nef cone E8D9}.
\end{proof}

\begin{lemma}
In the $E_7+E_7+A_3$ case $\operatorname{Lift}_{\geq 0}(X_0)$ has two chambers with one interior wall.
\end{lemma}

\begin{proof}
Recall that $h=(0,6l'-2\sum_{i=1}^7e_i'-\sum_{i=8}^{11}e_i')$ and $\xi=(-3l+\sum_{i=1}^7e_i,3l-\sum_{i=1}^{11}e_i')$ and it is quick to check that $\operatorname{Lift}_{\geq 0}(X_0)$ is spanned by $h$ and $h-2\xi$. Considering $h'=h+\epsilon\xi$ for $\epsilon >0$ we see that $h'. l <0$ which cannot be altered by flips, therefore we can only consider $h'=h-\epsilon\xi$ for $\epsilon >0$. Such an $h'$ is nef until $\epsilon >1$ where we flip the curves $e_i'$ for $i\geq 8$. We can then continue until $\epsilon =2$ where the stable model contracts the $V_1$ to a point. The wall and chamber decomposition is given in Figure \ref{nef cone E7E7A3}.
\end{proof}

\begin{lemma}\label{lifted polarization E8E8-4}
In the $E_8+E_8+\langle -4\rangle$ case $\operatorname{Lift}_{\geq 0}(X_0)$ has three chambers with two interior walls.
\end{lemma}

\begin{proof}
Recall that $h=(0,9l'-3\sum_{i=1}^8e_i'-2e_9'-e_{10}')$ and $\xi=(-3l+\sum_{i=1}^8e_i,3l'-\sum_{i=1}^{10}e_i')$ one can check that $\operatorname{Lift}_{\geq 0}(X_0)$ is spanned by $h$ and $h-3\xi$. Note that $h'=h+\epsilon\xi$ is not effective for $\epsilon >0$ as $h'. l <0$ and this cannot be altered by flipping curves. Therefore we need only to consider $h'=h-\epsilon\xi$ for $\epsilon >0$.
Such an $h'$ is nef until $\epsilon >1$ where we flip $e_{10}'$. After this flip, $h'$ is nef until $\epsilon >2$ where we flip $e_9'$. After this, $h'$ is nef until $\epsilon =3$ where the stable model contracts $V_1$ to a point. We therefore have three chambers and the decomposition is given in Figure \ref{nef cone E8E8}.
\end{proof}

We can now provide a proof of Theorem \ref{list of nef models}.
\renewcommand{\proofname}{Proof of Theorem \ref{list of nef models}}
\begin{proof}

Let $(X,H)$ be one of the standard divisor models, as constructed in Section \ref{sect:Explicit descriptions of families}. Let $(X',H')$ be another divisor model which gives rise to the same lattice as $(X,H)$, so that $X'\rightarrow \Delta$ is a Tyurin degeneration of K3 surfaces of degree 4, with polarizing line bundle $L'$ and $H'\in|(L')^{\otimes m}|$. Suppose that $X=X'$, by Lemma \ref{nef polarisations are related} we have that $H'\sim mH+nV_0$. Note also that by the definition of divisor models, $H'$ is nef on $X$. If $H'$ is ample then the the central fiber of the stable model $(\overline{X'},\overline{H'})$ is given in Table \ref{Complete Description Table}. If $H'$ is nef but not ample, then the central fiber of the stable model $(\overline{X'},\overline{H'})$ is given by contracting the appropriate curves of $X_0$.

Suppose that $X\neq X'$, then by a result of Shepherd-Barron \cite[Corollary 3.1]{SB83} there exists a sequence of flops $\phi:X'\dashrightarrow X$ such that $\phi_*H'$ is effective but not nef on $X$. By Lemma \ref{nef polarisations are related} we can write $\phi_*H'\sim mH+nV_0$ and the point $(m,n)$ does not belong to nef cone of $X$. We then perform the sequence of flops $\pi_{m,n}:X\dashrightarrow X_{m,n}$ defined as follows. In the appropriate subfigure of Figure \ref{fig:figures}, we start at the vertex $(1,0)$, move horizontally to the point $(m,0)$ and then move vertically to the point $(m,n)$. Moving horizontally does not cross any walls in our diagrams, but vertical moves can cross walls. One performs the flops corresponding to the wall-crossings arising from the vertical moves, we define $\psi:=\pi_{m,n}\circ\phi$ and so $\psi_*H'$ is nef on $X_{m,n}$. By Lemma \ref{sections on central fiber come from global} there exists $H_{m,n} \sim \psi_*H'$ which is effective and does not contain a stratum of any fibers. We can therefore use $H_{m,n}$ to construct stable models.

If $H_{m,n}$ is ample, $(X_{m,n}, H_{m,n})$ is the stable model of $(X',H')$ and the point $(m,n)$ belongs to the interior of a chamber. 
If instead $H_{m,n}$ is nef but not ample on $X_{m,n}$ then the point $(m,n)$ belongs to a boundary ray. Contracting the curves $C$ where $H_{m,n}.C=0$ gives rise to the stable model $(\overline{X_{m,n}},\overline{H_{m,n}})$ of $(X_{m,n},H_{m,n})$. The description of such curves is given in the corresponding Lemma in this section. By standard results of canonical models, we see that $(\overline{X_{m,n}},\overline{H_{m,n}})$ is the stable model of $(X',H')$. Hence this provides a description of all stable models of Tyurin degenerations of K3 surfaces of degree 4.
\end{proof}
\renewcommand{\proofname}{Proof}

We conclude this section with a proof of Corollary \ref{thm:lambda families intro}.

\begin{proof}[Proof of Corollary \ref{thm:lambda families intro}]

Let $(\mathcal{X}\rightarrow S_{\lambda},\mathcal{L})$ be a $\langle 4 \rangle$-quasipolarized nef $\lambda$-family of Tyurin degenerations. Let $(X\rightarrow \Delta,H)$ be a divisor model belonging to the $\lambda$-family so that $H\in|L|$ is an effective divisor where $L$ is a restriction of $\mathcal{L}$. Note that $X\rightarrow\Delta$ is a Tyurin degeneration of K3 surfaces of degree 4 and that $L$ is a multiple of the polarization line bundle, hence satisfying the conditions of Theorem \ref{list of nef models}. Furthermore, for each $\lambda$-family, at least one such divisor model exists corresponding to those constructed in Section \ref{sect:Explicit descriptions of families}. By Theorem \ref{list of nef models}, the stable model of $(X,H)$ belongs to Table \ref{Complete Description Table}.
\end{proof}

As a consequence of Theorem \ref{list of nef models} and Corollary \ref{thm:lambda families intro}, we have explicitly constructed $\langle 4 \rangle$-quasipolarized $\lambda$-families over tubular neighbourhoods of the Type II boundary components of $\overline{\mathcal{F}_4}^{tor}$. We have also shown that there is not necessarily a unique stable model associated to a given $\lambda$-family.

\bibliography{bib}

\begin{thebibliography}{AMRT75}

\bibitem[AE23]{AE23}
Valery Alexeev and Philip Engel.
\newblock Compact {M}oduli of {K}3 {S}urfaces.
\newblock {\em Ann. of Math.}, 198, 09 2023.

\bibitem[AMRT75]{AMRT75}
A.~Ash, D.~Mumford, M.~Rapoport, and Y.~Tai.
\newblock {\em Smooth compactification of locally symmetric varieties}, volume Vol. IV of {\em Lie Groups: History, Frontiers and Applications}.
\newblock Math Sci Press, Brookline, MA, 1975.

\bibitem[Anc87]{Ancona87}
Vincenzo Ancona.
\newblock Vanishing and nonvanishing theorems for numerically effective line bundles on complex spaces.
\newblock {\em Ann. Mat. Pura Appl. (4)}, 149:153--164, 1987.

\bibitem[BB66]{BB66}
W.~L. Baily and A.~Borel.
\newblock Compactification of arithmetic quotients of bounded symmetric domains.
\newblock {\em Ann. of Math.}, 84(3):442--528, 1966.

\bibitem[BR75]{BurnsRapoport}
Dan Burns, Jr. and Michael Rapoport.
\newblock On the {T}orelli problem for k\"ahlerian {$K-3$} surfaces.
\newblock {\em Ann. Sci. \'Ecole Norm. Sup. (4)}, 8(2):235--273, 1975.

\bibitem[BS76]{GlobalTheory}
Constantin B{\u a}nic{\u a} and Octavian St{\u a}n{\u a}{\c s}il{\u a}.
\newblock {\em Algebraic methods in the global theory of complex spaces}.
\newblock Editura Academiei, Bucharest; John Wiley \& Sons, London-New York-Sydney, 1976.
\newblock Translated from the Romanian.

\bibitem[Car85]{Carlson85}
James~A. Carlson.
\newblock The one-motif of an algebraic surface.
\newblock {\em Compositio Math.}, 56(3):271--314, 1985.

\bibitem[Dol12]{Dolg12}
Igor~V. Dolgachev.
\newblock {\em Classical Algebraic Geometry : a modern view}.
\newblock Cambridge University Press, Cambridge, 2012.

\bibitem[EVdV81]{EVdV81}
D.~Eisenbud and A.~Van~de Ven.
\newblock On the normal bundles of smooth rational space curves.
\newblock {\em Math. Ann.}, 256(4):453--463, 1981.

\bibitem[FM83]{FM81}
Robert Friedman and David~R. Morrison.
\newblock The birational geometry of degenerations: an overview.
\newblock In {\em The birational geometry of degenerations ({C}ambridge, {M}ass., 1981)}, volume~29 of {\em Progr. Math.}, pages 1--32. Birkh\"{a}user, Boston, MA, 1983.

\bibitem[Fri84]{Fried84}
Robert Friedman.
\newblock A new proof of the global torelli theorem for {K}3 surfaces.
\newblock {\em Ann. of Math.}, 120(2):237--269, 1984.

\bibitem[Har77]{Har77}
Robin Hartshorne.
\newblock {\em Algebraic geometry}.
\newblock Springer-Verlag, New York, 1977.
\newblock Graduate Texts in Mathematics, No. 52.

\bibitem[HT15]{HarThom15}
Andrew Harder and Alan Thompson.
\newblock The geometry and moduli of {K}3 surfaces.
\newblock In {\em Calabi-Yau varieties: arithmetic, geometry and physics}, pages 3--43. Springer, 2015.

\bibitem[Huy16]{Huy16}
Daniel Huybrechts.
\newblock {\em Lectures on K3 Surfaces}.
\newblock Cambridge Studies in Advanced Mathematics. Cambridge University Press, 2016.

\bibitem[KSC04]{KSC04}
J\'{a}nos Koll\'{a}r, Karen~E. Smith, and Alessio Corti.
\newblock {\em Rational and nearly rational varieties}, volume~92 of {\em Cambridge Studies in Advanced Mathematics}.
\newblock Cambridge University Press, Cambridge, 2004.

\bibitem[Laz16]{Laza16}
Radu Laza.
\newblock The {KSBA} compactification for the moduli space of degree two {$K3$} pairs.
\newblock {\em J. Eur. Math. Soc. (JEMS)}, 18(2):225--279, 2016.

\bibitem[LO19]{LOG19}
Radu Laza and Kieran O'Grady.
\newblock Birational geometry of the moduli space of quartic {$K3$} surfaces.
\newblock {\em Compos. Math.}, 155(9):1655--1710, 2019.

\bibitem[LO21]{LOG21}
Radu Laza and Kieran O'Grady.
\newblock G{IT} versus {B}aily-{B}orel compactification for {$K3$}'s which are double covers of {$\PS^1\times\PS^1$}.
\newblock {\em Adv. Math.}, 383:Paper No. 107680, 63, 2021.

\bibitem[PSS71]{PSS71}
I.~I. Pjatecki{\u i}-Shapiro and I.~R. Shafarevi{\v c}.
\newblock Torelli's theorem for algebraic surfaces of type {${\rm K}3$}.
\newblock {\em Izv. Akad. Nauk SSSR Ser. Mat.}, 35:530--572, 1971.

\bibitem[Rei97]{Reid96}
Miles Reid.
\newblock Chapters on algebraic surfaces.
\newblock In {\em Complex algebraic geometry ({P}ark {C}ity, {UT}, 1993)}, volume~3 of {\em IAS/Park City Math. Ser.}, pages 3--159. Amer. Math. Soc., Providence, RI, 1997.

\bibitem[SB83]{SB83}
N.~I. Shepherd-Barron.
\newblock Extending polarizations on families of {$K3$} surfaces.
\newblock In {\em The birational geometry of degenerations ({C}ambridge, {M}ass., 1981)}, volume~29 of {\em Progr. Math.}, pages 135--171. Birkh\"auser, Boston, MA, 1983.

\bibitem[Sca87]{Sca87}
Francesco Scattone.
\newblock On the compactification of moduli spaces for algebraic {$K3$} surfaces.
\newblock {\em Mem. Amer. Math. Soc.}, 70(374):x+86, 1987.

\bibitem[Sha80]{Shah80}
Jayant Shah.
\newblock A complete moduli space for {K}3 surfaces of degree {$2$}.
\newblock {\em Ann. of Math. (2)}, 112(3):485--510, 1980.

\bibitem[Sha81]{Shah81}
Jayant Shah.
\newblock Degenerations of {K}3 surfaces of degree {$4$}.
\newblock {\em Trans. Amer. Math. Soc.}, 263(2):271--308, 1981.

\bibitem[Sil09]{Sil09}
Joseph~H. Silverman.
\newblock {\em The arithmetic of elliptic curves}, volume 106 of {\em Graduate Texts in Mathematics}.
\newblock Springer, Dordrecht, second edition, 2009.

\bibitem[Ura84]{Ura84}
Tohsuke Urabe.
\newblock On quartic surfaces and sextic curves with singularities of type {$\tilde E_8,\;T_{2,3,7},\;E_{12}$}.
\newblock {\em Publ. Res. Inst. Math. Sci.}, 20(6):1185--1245, 1984.

\bibitem[Ura86]{Ura86}
Tohsuke Urabe.
\newblock Classification of nonnormal quartic surfaces.
\newblock {\em Tokyo J. Math.}, 9(2):265--295, 1986.

\end{thebibliography}
\end{document}